\documentclass[letterpaper,11pt,reqno]{amsart} 
\usepackage[portrait,margin=1.25in]{geometry} 
\usepackage{style}

\begin{document}
\title[$\ell^{p}$ Curie--Weiss model]{Curie--Weiss Model under $\ell^{p}$ constraint and a Generalized Hubbard--Stratonovich Transform}

\author[Dey]{Partha S.~Dey$^1$}
\address{$^1$University of Illinois at Urbana Champaign, 1409 W Green Street, Urbana, Illinois 61801.}
\email{psdey@illinois.edu}

\author[Kim]{Daesung Kim$^2$}
\address{$^2$Georgia Institute of Technology, 686 Cherry Street, Atlanta, GA 30332.}
\email{dkim3009@gatech.edu}
		
\begin{abstract}
We consider the Ising Curie--Weiss model on the complete graph constrained under a given $\ell^{p}$ norm for some $p>0$. For $p=\infty$, it reduces to the classical Ising Curie--Weiss model. We prove that for all $p>2$, there exists $\gb_{c}(p)$ such that for $\gb<\gb_{c}(p)$, the magnetization is concentrated at zero and satisfies an appropriate Gaussian CLT. In contrast, for $\gb>\gb_{c}(p)$ the magnetization is concentrated at $\pm m_\ast$ for some $m_\ast>0$. We have $\gb_{c}(p)>1$ for $p>2$ and $\lim_{p\to\infty}\gb_{c}(p)=3$. We further generalize the model for general symmetric spin distributions and prove a similar phase transition. For $0<p<1$, the log-partition function scales at the order of $n^{2/p-1}$. The proofs are based on a generalized Hubbard--Stratonovich (GHS) transform, which is of independent interest. 
\end{abstract}
		
\date{\today}
\maketitle

\setcounter{tocdepth}{1}
\tableofcontents
	
\section{Introduction}\label{sec:intro}

\subsection{The Model}\label{ssec:model}

We consider the mean-field Curie--Weiss model under an $\ell^p$ constraint. The model has the same formulas for the Hamiltonian and the Gibbs measure as in the classical Curie--Weiss model. However, the spins in this model are taken from the $n$-dimensional $\ell^p$ sphere instead of the hypercube as in the classical case. 

We fix a positive real number $p>0$ and define the spin configuration space $S_{n,p}$ to be the $n$-dimensional $\ell^{p}$-sphere of radius $n^{1/p}$, \ie
\begin{align*}
S_{n,p}:=\left\{\mvgs\in\dR^{n}\ \biggl|\ \frac1n\sum_{i=1}^{n}\abs{\gs_i}^p=1\right\}.
\end{align*}
Consider the Gibbs measure at inverse temperature $\gb>0$ given by
\begin{align}\label{eq:magnetization}
d\mu_{n,\gb,p}(\mvgs):=\cZ_{n,p}(\gb)^{-1}\exp(\gb H_{n}(\mvgs))\, d_{n,p}\mvgs,
\end{align}
where $d_{n,p}$ is the uniform measure on $S_{n,p}$, $H_n(\mvgs):=\frac1n\sum_{i<j}\gs_i \gs_j$ is the Hamiltonian, and
\begin{align}
\cZ_{n,p}(\gb)&:=\int_{S_{n,p}} \exp(\gb H_{n}(\mvgs))\, d_{n,p}\mvgs\label{def:pf}
\end{align}
is the partition function. 

The classical Curie--Weiss model can be considered as the $p=\infty$ version of this model, where the critical inverse temperature is $\gb_{c}(\infty)=1$. When $p=2$, this model is equivalent to the self-scaled version with standard Gaussian base measure, and it is known that a similar phase transition happens at $\gb_{c}(2)=1$ (see~\cite{CG16, G14} for more details at the critical $\gb_c(2)$).

\subsection{Motivation}

The classical Curie--Weiss model and its generalization by~\cite{EN78} consider the product base measure, which leads to the independent structure. A natural question is whether there is a similar phase transition when the spin configuration has a dependence structure. In particular, in the $\ell^p$ constraint model, each spin has a global interaction with other spins and one question is whether there exists some $\gb_{c}(p)\in (0,\infty)$ such that for $\gb<\gb_{c}(p)$ the typical magnetization is concentrated at $0$ and for $\gb>\gb_{c}(p)$ it is not. Another question is whether developing a general framework to analyze constrained spin models is possible. 

The classical model can be understood as the special case $p=\infty$ in our model, and we are interested in the asymptotic behavior when $n$ goes to infinity for finite $p>0$. A natural question is whether the limits in $n$ and $p$ are interchangeable and whether there is a discontinuity. Specifically, we are interested in the behavior of the function $p\mapsto \gb_{c}(p)$ as $p$ varies from $0$ to $\infty$, for example, whether it is continuous and type of discontinuity if it is not. We clearly know that $\gb_{c}(2)=\gb_{c}(\infty)=1$. One might suspect that $\gb_{c}(p)$ is a constant or a continuous function for $p\in[2,\infty]$. However, the actual picture is different and needs detailed analysis.

Another motivation of the $\ell^p$-constraint model is the connection with the $\ell^p$-Gaussian Grothendieck problem. The $\ell^p$-Grothendieck problem is the optimization problem regarding the quadratic forms of a matrix over $\ell^p$. The corresponding disorder model of the $\ell^p$ Curie--Weiss model is closely related to the $\ell^p$-Grothendieck problem with a random matrix (see~\cite{ChenSen23, Domin22, KhoNao12}).

\subsection{Simplification}\label{ssec:gen}

We can rewrite the model~\eqref{eq:magnetization} with respect to a sequence of i.i.d.~random variables to remove the dependency arising from the $\ell^{p}$ constraint, allowing us to generalize the model further. We define the norm $\norms{\cdot}_{n,p}$ on $\dR^{n}$ as follows
\begin{align}\label{def:norm}
\norms{\mvx}_{n,p}:=\left( \frac1{n}\sum_{i=1}^{n}\abs{x_{i}}^{p}\right)^{1/p}\text{ for } \mvx\in \dR^{n}.
\end{align}
Next, we define a collection of symmetric probability densities indexed by $p>0$ given by
\begin{align*}
\rho_{p}(x):=c_{p}\cdot e^{-\abs{x}^p/p},\quad x\in\dR
\end{align*} 
where $c_{p}:=\frac{p^{-1/p}}{2\gC(1+1/p)}$
for $p>0$ and use $Z_{p}$ to denote a typical random variable having density $ \rho_{p}$. Note that $|Z_{1}|\sim \text{Exp}(1)$ and $Z_{2}\sim \N(0,1)$. Moreover, it is easy to check that 
\[
\E\abs{Z_p}^{p}=1 \text{ for all }p>0.
\]
One can easily verify that $\abs{Z_p}^{p}/p$ is distributed as a Gamma random variable with shape parameter $1/p$ and scale parameter $1$, enabling us to compute all of its polynomial moments; in particular, the variance of $Z_p$ can be explicitly written as 
\begin{align*}
\nu_{p}^{2}:=\var(Z_{p})= \E Z_{p}^{2}=p^{2/p}\cdot \frac{\gC(3/p)}{\gC(1/p)} \quad\text{ for } p>0.
\end{align*}
Let the coordinates of the random vector $\mvX=(X_{1},X_{2},\ldots,X_{n})$ be $n$ i.i.d.~random variables with density $\rho_{p}$. Define 
\begin{align*}
\widehat{X}_{i}:=\frac{X_i\;\;\; }{\norms{\mvX}_{n,p}},\quad i=1,2,\ldots,n.
\end{align*}
It is well-known (see~\cite{ScheZinn, BGMN05}) that the vector $\widehat\mvX=(\widehat{X}_{1},\widehat{X}_{2},\ldots,\widehat{X}_n)$ is uniformly distributed on the $n$-dimensional $\ell^p$-sphere $S_{n,p}$. In particular, for the partition function~\eqref{def:pf}, we have
\begin{align}\label{eq:Zrhop}
\cZ_{n,p}(\gb) &=\E\exp(\gb H_n(\widehat{\mvX}))\nonumber\\ &= \int_{\dR^{n}} \exp\left(\frac{\gb n}{2}\cdot \frac{\abs{n^{-1}\sum_{i=1}^nx_i}^2}{\norms{\mvx}_{n,p}^2} - \frac{\gb}{2}\cdot \frac{n^{-1}\sum_{i=1}^{n}x_i^2}{\norms{\mvx}_{n,p}^2} \right) \prod_{i=1}^{n}\rho_{p}(x_{i})\,dx_{i}.
\end{align}

When $p\ge 1$, clearly 
\begin{align*}
\frac{\abs{n^{-1}\sum_{i=1}^{n}x_i}}{\norms{\mvx}_{n,p}}\le 1 \text{ for all } \mvx \in\dR^{n}.
\end{align*}
In particular, ${H_n(\mvgs)}\le n/2$ and for the log-partition function we get $\log \cZ_{n,p}(\gb) \le n\gb/2$ for all $\gb>0$. However, for $p\in(0,1)$ the Hamiltonian grows at the order of $n^{2/p-1}\gg n$ when finitely many of the spins are $\Theta(n^{1/p})$. The representation~\eqref{eq:Zrhop} also allows us to generalize the model further by replacing the base measure $\rho_p(x)dx$ with a general symmetric probability measure $d\rho(x)$; we postpone the discussion to Section~\ref{sec:ssm}.

By the strong law of large numbers, when $X_1, X_2,\ldots, X_n$ are i.i.d.~from $\rho_p$, we have 
\begin{align*}
\norms{\mvX}_{n,p}^{p} \text{ is concentrated around } \E|X_1|^p=1 \text{ for large } n.
\end{align*} 
Thus, motivated by the classical results by~\cite{EN78}, one expects  the critical inverse temperature to be  $\gb_{c}(p)$, where
\begin{align}\label{eq:betacp}
\gb_c(p):=\nu_{p}^{-2} =  \frac{3}{p^{2/p}}\cdot\frac{\gC(1+1/p)}{\gC(1+3/p)}
\end{align}
for $p\in [1,\infty)$. In Theorem~\ref{thm:rhop} and Theorem~\ref{thm:genrho}, we prove that $\gb_c(p)$ is indeed the critical inverse temperature for the $\ell^p$ constraint Curie--Weiss model and its generalization to the $\ell^p$ self-scaled model (defined in Section~\ref{sec:ssm}) for $p\ge 2$ (see also Subsection~\ref{ssec:critical}). Note that $\gb_c(2)$ coincides with the classical critical inverse temperature,\ie $\gb_c(\infty)=1=\gb_c(2)$. 

However, $\gb_c(p)$ converges to $3=1/\var(U([-1,1]))$ as $p\to\infty$, which is different from $\gb_{c}(\infty)=1=1/\var(U(\{+1,-1\}))$. Here, $U([-1,1])$ and $U(\{+1,-1\})$ are uniform distributions over $[-1,1]$ and $\{+1,-1\}$ respectively. This is due to the fact that $S_{n,p}$ can converge to the discrete hypercube or the continuous  $\ell^{\infty}$-sphere depending on how we take the limit $p\to\infty$. For fixed $n$, if we take the limit $p\to \infty$ on the approximate constraint $\left(\sum_{i=1}^n |x_i|^p\right)^{1/p}\approx n^{1/p}$, then the limiting constraint defines the $\ell^\infty$-sphere $\{\mvx\in\dR^n:\max_i |x_i|\approx 1\}$. On the other hand, for fixed $n$ if we take the limit $p\to\infty$ on the exact constraint $\sum_{i=1}^n |x_i|^p=n$, then we obtain the discrete hypercube 
\[
\left\{\mvx\in\dR^n:\sum_{i=1}^n\lim_{p\to\infty}|x_i|^p=n\right\}=\{\mvx\in\dR^n:|x_i|=1\text{ for all }i=1,2,\ldots,n\}.
\]
Here, we do not pursue a detailed analysis of the transition at $p=\infty$ and put it as an open question.

\subsection{Main Results}\label{ssec:mainres}

We first describe the main results before discussing the heuristics, background literature, and ideas behind the proofs. In Section~\ref{ssec:lpres}, we state our main results about the $\ell^p$ constraint model for general $p>0$. For $p\ge 2$, fluctuation results for the magnetization are similar to the classical Curie-Weiss model. However, a direct application of the existing tools does not work in this model due to the structure of the $\ell^p$ self-scaled Hamiltonian.  In Section~\ref{ssec:ghstres}, we develop the generalized Hubbard--Stratonovich (GHS) transform based on the existence of a particular additive process with independent increments. GHS transform plays the leading role in analyzing the $\ell^p$ contained model. In Section~\ref{sec:ssm}, we extend the analysis to a general class of $\ell^p$ self-scaled model in a similar spirit to~\cite{EN78}.

\subsubsection{$\ell^p$ constraint model}\label{ssec:lpres}

Our first main result is a phase transition of the Curie--Weiss model under $\ell^p$ constraint. Let $m_{n,\gb,p}(\mvgs)=\frac{1}{n}\sum_{i=1}^n \gs_i$ be the magnetization where $\mvgs$ is chosen according to the Gibbs measure~\eqref{eq:magnetization} and $\gb_c(p)$ be defined as~\eqref{eq:betacp}. The cumulant generating function of $Z_p\sim\rho_p$ is given by 
\begin{align}\label{def:psip}
\psi_p(t):=\log\E[e^{tZ_p}],\quad t\in\dR \, \text{ where }Z_p\sim \rho_p.
\end{align}

First, we have the variational representation of the limiting free energy for arbitrary $p\ge 2$.

\begin{thm}\label{thm:rhop-partition}
Let $p\in[2,\infty)$ be fixed and $\gb_c(p)$ be as defined in \eqref{eq:betacp}. For all $\gb\ge 0$, we have
\begin{align}\label{eq:thm1}
\begin{split}
\lim_{n\to\infty}&\frac1n \log \cZ_{n,p}(\gb)\\
&=\sup_{z,w\ge 0}\left\{\psi_{p}\left(\sqrt{\gb}{zw}{(1+z^p)^{-1/p}}\right)-\frac{1}{p}\log(1+z^p)-\frac{p-2}{2p}w^{2p/(p-2)}\right\}\\
&= \sup_{0\le u< 1,w\ge 0 }\left\{\psi_{p}\left(\sqrt{\gb}\cdot uw  \right)+\frac{1}{p}\log(1-u^{p})-\frac{p-2}{2p} w^{2p/(p-2)}\right\}. 
\end{split}
\end{align}
\end{thm}

In Theorem~\ref{thm:genrho}, we present a general version of this result. The following result says that for $p\ge 2$, we have a phase transition of the magnetization $m_{n,\gb,p}$ at $\gb=\gb_c(p)$ similar to the classical model.

\begin{thm}\label{thm:rhop}
Let $p\in[2,\infty)$ be fixed and $\gb_c(p)$ be as defined in \eqref{eq:betacp}. 
\begin{enumeratea}
\item    
If $\gb<\gb_c(p)$, then $m_{n,\gb,p}$ is concentrated at zero and has a Gaussian limit as $n\to\infty$. More precisely, 
\[
\sqrt{n}\cdot m_{n,\gb,p}=\frac{1}{\sqrt{n}}\sum_{i=1}^n \gs_i  \implies \mathrm{N}(0,(\gb_c(p)-\gb)^{-1})
\]
as $n\to\infty$ and $\cZ_{n,p}(\gb)$ is uniformly bounded in $n$. In particular,
\[
\lim_{n\to\infty}\frac1n \log\cZ_{n,p}(\gb)=0.
\]
\item
If $\gb=\gb_c(p)$, then 
\[
n^{1/4}m_{n,\gb,p}(\mvgs) \implies A^{-1/4}\cdot Z_4,
\]
where
\begin{align*}
    A = \gb_c^2\left(\frac2{p}  + \frac{\kappa_p}{6}\right), \quad
    \kappa_p: = 3- \E Z_p^4/(\E Z_p^2)^2= 3-\Gamma(5/p)\gC(1/p)\Gamma(3/p)^{-2}\ge 0.
\end{align*}

\item
If $\gb>\gb_c(p)$, then 
\[
\lim_{n\to\infty}\frac1n \log\cZ_{n,p}(\gb)>0,
\]
and $m_{n,\gb,p}$ is concentrated at $\pm m_\ast$ where $m_\ast =\frac{w_\ast}{\sqrt{\gb}}$ and $w=w_\ast$ is the unique positive solution to the equation
\begin{align*}
\psi_p'\left(\frac{\sqrt{\gb}w\ }{(1+w^2)^{1/p}}\right)
=\frac{1}{\sqrt{\gb}}w(1+w^2)^{1/p}.
\end{align*}
\end{enumeratea}
\end{thm}
\begin{rem}
If $p=2$ and $\gb>\gb_c(2)=1$, then $\psi_2'(t)=t$ and so $w_\ast=\sqrt{\gb-1}$ and $m_\ast=\sqrt{1-1/\gb}$. However, for general $p>2$, there is no explicit formula for $m_*$. It will also be interesting to find rates of convergence for the above distributional limits.
\end{rem}

The main tool used in the proof of Theorems~\ref{thm:rhop-partition} and~\ref{thm:rhop} is the GHS transform introduced in Section~\ref{ssec:ghstres}. As an application of the GHS transform (see Lemma~\ref{lem:mgfbd}), we have for all $p\ge 2, t\in\dR$, 
$
\psi_p(t)\le \psi_2(t\nu_p) = t^2\nu_p^2/2.
$
Using Young's inequality $xy\le \frac2px^{p/2}+\frac{p-2}py^{p/(p-2)}$ for $p>2,x,y\ge0$, it is an easy exercise to see that for $\gb\le \gb_c(p)=1/\nu_p^2$, we have 
\begin{align*}
\psi_{p}\left(\sqrt{\gb}\cdot uw\right) &+\frac{1}{p}\log(1-u^{p})-\frac{p-2}{2p}w^{2p/(p-2)}\\
&\le 
\frac12u^2 w^{2}+\frac{1}{p}\log(1-u^{p})-\frac{p-2}{2p}w^{2p/(p-2)} \\
&\le \frac{1}{p}u^{p} +\frac{1}{p}\log(1-u^{p})\le 0
\end{align*}
for all $w\ge 0, u\in[0,1)$. However, it is nontrivial to show that the supremum in equation~\eqref{eq:thm1} is indeed strictly positive for $\gb>\gb_c(p)$. The proof is presented in Section~\ref{sec:lowtemp}.

The GHS transform allows us to describe the typical spin configuration under the Gibbs measure by introducing hidden random variables. 
Note that a similar analysis for typical configuration of the classical Curie-Weiss model can be found in~\cite{Pap89}.

\begin{thm}\label{thm:typical}
Let $\gb<\gb_c$ and $p\ge 2$. Let $(X_1,X_{2,}\ldots, X_n)$ be random variables under the Gibbs measure 
\begin{align}\label{eq:Gibbsmod}
\frac{1}{\wh{\cZ}_{n,p}(\gb)}
\exp\left(\frac{\gb n}{2}\cdot \frac{\abs{n^{-1}\sum_{i=1}^nx_i}^2}{\norms{\mvx}_{n,p}^2} \right) \prod_{i=1}^{n}\rho_{p}(x_{i})\,dx_{i}
\end{align}
where $\wh{\cZ}_{n,p}(\gb)$ is the normalizing constant. Then, there exist random variables $(Z^{(n)}, U^{(n)})$ such that conditioned on $(Z^{(n)},U^{(n)})=(z,u)$, $(X_1,X_{2},\ldots,X_n)$  is an i.i.d.~sequence of random variables and 
\[
\frac1{\sqrt{n}}\sum_{i=1}^n X_i - \sqrt{\gb/\gb_c^2}\cdot zu \convd \N(0,1/\gb_{c}).
\]
Furthermore, $Z^{(n)}U^{(n)}\convd \N(0,\gb_c/(\gb_c-\gb))$ as $n\to\infty$. Consequently, unconditionally, $\frac{1}{\sqrt{n}}\sum_{i=1}^n X_i \convd \N(0, 1/(\gb_c-\gb))$.
\end{thm}

\begin{rem}\label{rem:typical}
Let $\gb\le \gb_c(p)$ and $p\ge 2$. The typical behavior of $\mvgs=(\gs_1,\gs_2,\ldots,\gs_n)$ under the Gibbs measure~\eqref{eq:magnetization} and that of $\mvX=(X_1,X_2,\ldots,X_n)$ under~\eqref{eq:Gibbsmod} is not exactly the same because $\mvX$ is not scaled by the $\ell^p$ norm of $\mvX$ and the Gibbs measures are slightly different. However, one can see that their asymptotic behaviors agree for the following reasons. First, the difference between the Gibbs measures is the term 
\[
\frac{\gb}{2}\cdot \frac{n^{-1}\sum_{i=1}^{n}x_i^2}{\norms{\mvx}_{n,p}^2}
\]
in the exponent. This term is upper bounded by $\gb/2$ for $p\ge 2$ and converges to a constant almost surely as $n\to\infty$. Modulo this difference, we have $\mvgs=\mvX/\norms{\mvX}_{n,p}$. Since $\mvX/\norms{\mvX}_{n,p}$ and $\norms{\mvX}_{n,p}$ are independent and $\norms{\mvX}_{n,p}$ converges to 1 in probability, we see that the finite-dimensional marginals of $\mvgs$ and $\mvX$ (and also their magnetizations) have the same asymptotic behavior.
\end{rem}

Here, we note that the result in Theorem~\ref{thm:rhop-partition} cannot hold for $p<2$, as the right-hand side of equation~\eqref{eq:thm1} is infinity for $p<2$. However, in the $p>1$ case, we can use the theory of self-normalized processes (see~\cite{PLS}) to get the following variational representation of the limiting free energy.

\begin{prop}\label{prop:selfnormal}
For $p>1$, we have
\begin{align}\label{eq:sn}
\begin{split}
&\lim_{n\to\infty}\frac1n \log\cZ_{n,p}(\gb) \\
&\qquad=\sup_{y,c\ge 0}\inf_{t\ge 0 }\left\{ \frac{\gb y^2}{2}+ \psi_{p}\left(\frac{ct}{(1+ty)^{1/p}}\right)- \frac{1}{p}\log(1+ty)-\frac{p-1}{p}ty c^{p/(p-1)} \right\}.
\end{split}
\end{align}
\end{prop}
For $p\neq 2$, the right-hand side of equation~\eqref{eq:sn} in   Proposition~\ref{prop:selfnormal} can be written in a way similar to equation~\eqref{eq:thm1}. By doing a change of variables, $w=(\gb y^2)^{1/2-1/p}, z =(\gb y^2)^{1/p}c^{-1/(p-1)}, t = s\cdot z^p/y$, we get
\begin{align*}
&\lim_{n\to\infty}\frac1n \log\cZ_{n,p}(\gb) \\
&\qquad =\sup_{z,w\ge 0}\inf_{s\ge 0 }\left\{ \psi_{p}\left(\frac{\sqrt{\gb}\cdot szw}{(1+sz^p)^{1/p}}\right)- \frac{1}{p}\log(1+sz^p)-\frac{2ps-2s-p}{2p} w^{2p/(p-2)} \right\}.
\end{align*}
By taking $s=1$, we thus get for $p> 1$,
\begin{align*}
\lim_{n\to\infty}\frac1n \log\cZ_{n,p}(\gb) 
\le \sup_{z,w\ge 0}\left\{ \psi_{p}\left(\frac{\sqrt{\gb}\cdot zw}{(1+z^p)^{1/p}}\right)- \frac{1}{p}\log(1+z^p)+\frac{2-p}{2p} w^{2p/(p-2)} \right\}.
\end{align*}
For $p\in(1,2)$, we can use the GHS transform to get the following result. Since it follows from Theorem~\ref{thm:genrho} with a change of variables, we omit the proof. 

\begin{thm}\label{thm:rhop12}
Let $p\in (1,2)$ be fixed and $F_p(\gb):=\lim_{n\to\infty}\frac1n\log\cZ_{n,p}(\gb)$ for $\gb\ge0$. We have 
\[
\sup_{z\ge 0 }\left(\psi_{p}(\sqrt{y}z/(1+|z|^p)^{1/p})-\frac1p\log(1+|z|^p)\right)
=\sup_{x\ge 0}\left\{F_p(x y )-\frac{2-p}{2p}x^{\frac{p}{2-p}}\right\}.
\]
In particular, we have
\begin{align}\label{eq:limitbound}
\lim_{n\to\infty}\frac1n\log\cZ_{n,p}(\gb)
\le \inf_{w> 0 }\sup_{z\ge 0 }\left\{\psi_{p}\left(\frac{\sqrt{\gb}\cdot zw}{(1+z^p)^{1/p}}\right)-\frac1p\log(1+|z|^p)+\frac{2-p}{2p}w^{\frac{2p}{p-2}}\right\}.    
\end{align}
\end{thm}

\begin{rem}
Note that the right-hand side of~\eqref{eq:limitbound} is not a convex function of $\gb$, in general, because it is defined by an infimum of convex functions. Since the limiting free energy is convex in $\gb$, the inequality~\eqref{eq:limitbound} may not be a tight bound for the limiting free energy. Proposition~\ref{prop:selfnormal} tells us that the log-partition function converges as $n\to\infty$ to a well-defined function of $\gb$. However, the limit is given by supremum and infimum over three variables, so the variational formula does not give us much information about the phase transition and the critical temperature. It is open to find a variational formula for the limiting free energy similar to the $p\ge 2$ case and the critical temperature for $1<p<2$. We believe that similar to the $p\ge2$ case, a phase transition happens in the $p\in(1,2)$ case at $\gb=\gb_c(p)$.
\end{rem}

When $0<p<1$, the log-partition function has a super-linear scaling, and the limiting free energy is an explicit function of $p$. 

\begin{thm}\label{thm:rhop2}
Define  $p_1:=0$ and $p_k:=\frac{2\log(1+1/k)}{\log(1+2/(k-1))}$ for $k\ge 2$. For $p\in (p_{k-1},p_k]$, define $\tau(p):=k^{1-2/p}\cdot (k-1)$. For $p\in(0,1)$, we have
\[
\lim_{n\to\infty}n^{1-2/p}\log \cZ_{n,p}(\gb) = \frac{\gb}{2}\tau(p).
\]
\end{thm}
The result follows from Theorem~\ref{thm:genrho2} as a particular case with $X\sim\rho_p$. 

\subsubsection{Generalized Hubbard--Stratonovich transform}\label{ssec:ghstres}

The following result plays a crucial role in analyzing the $\ell_{p}$ constraint Curie--Weiss model. In the classical Curie--Weiss model, the phase transition and the Gaussian fluctuation can be investigated by the Hubbard--Stratonovich transform, which enables us to linearize the Hamiltonian (see Section~\ref{ssec:classicCW}). It turned out that we can use the following generalized version of the Hubbard--Stratonovich (GHS) transform to linearize the $\ell^p$ self-scaled Hamiltonian in our model.

\begin{thm}[GHS transform]\label{thm:ghst}
For any $p\ge 2$, there exists a positive r.v.~$U=U_{p,2}$, such that for all $x\in\dR,y>0$ we have
\[
y^{-1/p}\exp\left(x^2y^{-2/p}\right)
= c_{p} \cdot\E_U \int_\dR e^{\sqrt{2}xzU - y\abs{z}^p/p}\,dz
\]
where $c_{p}=\frac{p^{-1/p}}{2\gC(1+1/p)}$. 
\end{thm}

\begin{cor}
In particular, for $(\go_i,\nu_i)_{i\ge 1}$ i.i.d.~random vectors with $\nu_1>0$ a.s.~and $p\ge 2$, we have
\begin{align*}
&\E\left[\biggl(\frac1n\sum_{i=1}^n \nu_i\biggr)^{-1/p} \exp\left(\frac{\gb }{2n}\biggl(\sum_{i=1}^n \go_i\biggr)^2\biggl(\frac1n\sum_{i=1}^n \nu_i\biggr)^{-2/p}\right)\right]\\
&\qquad\qquad = c_{p}\cdot\E_U \int_\dR \left(\E \exp\left(\sqrt{\frac{\gb}{n}}\cdot\go_{1} z U - \frac{\nu_1}n\cdot \frac1p\abs{z}^p\right)\right)^n \,dz.
\end{align*}
\end{cor}

The key ingredient of the GHS transform is the existence of a positive random variable $U$ independent of $X\sim\rho_p$, $p\ge 2$, satisfying $X\cdot U\sim \mathrm{N}(0,1)=\rho_2$. Indeed, we construct a stochastic process $U_{p,q}$ indexed by $0<q<p$ such that $Z_p\cdot U_{p,q}\sim \rho_q$ where $Z_p\sim \rho_p$. The proof is based on the following construction of an additive process.

\begin{thm}\label{thm:lgamma}
There exists an additive process $(Y_t)_{t>0}$ with independent increments such that $Y_t\equald t\log X_t$ and $X_t\sim\textsc{Gamma}(t,1)$ for all $t>0$. Moreover, $Y_0\sim -\textsc{Exp}(1)$. 
\end{thm}

By Theorem~\ref{thm:lgamma} and letting $U_{p,q} \equald \exp(Y_{1/q}-Y_{1/p})$, one can show the following.

\begin{prop}\label{prop:decomp}
Let $Z_p\sim \rho_{p}$ for $p>0$. For every $0<q<p$, there exists a positive r.v.~$U=U_{p,q}$ independent of $Z_{p}$ such that 
\[
Z_{p}\cdot U_{p,q} \equald Z_{q}.
\]
Moreover, $U_{p,q}$ has a strictly positive density $\theta=\theta_{p,q}$ on $(0,\infty)$ which satisfies
\begin{align*}
g(x):=\frac{\theta(x)}{x^{\ga/2}\exp(-x^{\ga}/\ga)}\to C_{p,q}:= \frac{p c_q}{c_p\sqrt{2\pi(p-q)}} \text{ as $x\to\infty$},
\end{align*}
where $\ga=pq/(p-q)$.
\end{prop}

\begin{rem}
In the case of $p=4$, one can explicitly compute the density of $U_{4,2}$ such that $ Z_{4}\cdot U_{4,2}\sim \N(0,1)$, where $Z_{4}\sim \rho_{4}$ with 
$
\rho_{4}(x) = \sqrt{2} \exp(-x^{4}/4)/\gC(1/4), x\in\dR.
$
Let $\theta_{4,2}$ be the density given by
$
\theta_{4,2}(u):=\sqrt{2}u^2 \exp(-u^{4}/4)/\gC(3/4), u>0.
$
One can check that if $\theta_{4,2}$ is the density of  $U_{4,2}$ and $U_{4,2}\perp Z_{4}$, then $U_{4,2}\cdot Z_{4}\sim \N(0,1)$. 
\end{rem}

Proposition~\ref{prop:decomp} says that a random variable $Z_q$ can be written as a product of two independent random variables. It is well-known (see~\cite{Duf10} for example) that there are various identities regarding families of  Gamma and Beta random variables, for example,
\[
G^{(a)}+G^{(b)}=G^{(a+b)},\qquad
B^{(a,b)}G^{(a+b)}=G^{(a)}
\]
where $G^{(a)}$ is a Gamma random variable with the shape parameter $a$ and the scale parameter 1, and $B^{(a,b)}$ is a Beta random variable with parameters $a,b$. This gives rise to the so-called beta-gamma algebra. Since the $p$-th power of $Z_p\sim\rho_p$ is Gamma distributed, the proposition can be understood as a variant of the beta-gamma algebra for the powers of Gamma distributions.

For general $p>q>0$, using Mellin transform and the facts that $\abs{Z_{p}}^p/p\sim \text{Gamma}(1/p,1)$, $\abs{Z_p}\cdot U_{p,q}\equald \abs{Z_q}$, one can check that the random variable $p^{1/p}q^{-1/q}\cdot U_{p,q}$ has density given by
\[
\frac{1}{2\pi \i}\int_{1-\i\infty} ^{1+\i\infty}x^{-z}\cdot  \frac{\gC(1/p)}{\gC(1/q)} \frac{\gC(z/q)}{\gC(z/p)}\,dz,\text{ for } x>0.
\]
Note that, for $a>0$, we have
\[
\abs{\gC(a+\i b)}/\gC(a) = \prod_{k=0}^{\infty}\left(1+b^{2}/(a+k)^2\right)^{-1/2}\approx \exp(-b^{2}/2a),
\]
which implies that the density is well-defined for $p>q$. The asymptotic behavior of the density also follows from the relation $\abs{Z_{p}}\cdot U_{p,q}\equald \abs{Z_{q}}$.  Using the density formula for the product of two independent rvs, we get for  any $x>0$,
\begin{align*}
{c_q}/{c_p}
&
=\int_0^\infty \exp\left({x^q}/{q}-{(x/y)^p}/{p}\right)\cdot {\theta(y)}/{y}\cdot dy\\
&= \int_0^\infty y^{\ga/2-1}\exp\left(x^q/q-(x/y)^p/p-y^{\ga}/\ga\right)\cdot g(y)\, dy
\end{align*}
where $\ga=pq/(p-q)$. Doing a  change of variable $y=x^{(p-q)/p}\cdot (1+tx^{-q/2}), x=\eps^{-2/q}$, we get for all $\eps>0$,
\begin{align*}
\int_{-1/\eps}^\infty \exp(-\eps^{-2}h(t\eps))\cdot 
(1+t\eps)^{\ga/2-1}g\left(\eps^{-2/\ga}(1+t\eps)\right)
\, dt = {c_q}/{c_p}
\end{align*}
where 
$
h(u):=(1+u)^{-p}/p+(1+u)^{\ga}/\ga-1/q.
$
Note that,  $h(0)=h'(0)=0$ and $h''(u)\ge h''(0)=p^2/(p-q)$ for all $u\ge 0$. Thus, we have
\begin{align*}
2u^{-2}h(u)\ge \lim_{u\to 0}2u^{-2}h(u)={p^2}/(p-q).
\end{align*}
Using DCT with $\eps\downarrow 0$ and the Gaussian integral formula, thus we expect that
\[
\sqrt{2\pi (p-q)/p^{2}}\cdot \lim_{x\to\infty} g(x)=c_{q}/c_{p}.
\]
A rigorous argument is presented in Section~\ref{sec:density} in the proof of Proposition~\ref{prop:decomp}.

In the proofs of the main results, we mainly used the existence of a random variable $U_{p,2}$ independent of $Z_{p}$ in Proposition~\ref{prop:decomp} that satisfies $Z_{p}\cdot U_{p,2}\sim \N(0,1)$, whereas we indeed constructed a stochastic process $U_{p,q}$  in $p$ for fixed $q$  in Theorem~\ref{thm:lgamma}. A geometric interpretation and other possible applications of this stochastic process are of independent interest.

\subsubsection{The $\ell^p$ self-scaled model}\label{sec:ssm}

One can generalize the $\ell^p$ constraint model to the $\ell^p$ self-scaled model. As we have seen, the $\ell^p$ constraint Curie--Weiss model  can be written in terms of self-normalized process $\widehat\mvX=(\widehat{X}_{1},\widehat{X}_{2},\ldots,\widehat{X}_n)$ where
\begin{align*}
\widehat{X}_{i}:=\frac{X_i\;\;\; }{\norms{\mvX}_{n,p}},\quad i=1,2,\ldots,n,
\end{align*}
and the random vector $\mvX=(X_{1},X_{2},\ldots,X_{n})$ be $n$ i.i.d.~random variables from $\rho_{p}$. Motivated by~\eqref{eq:Zrhop}, one can consider $\ell^p$ self-scaled model with a symmetric base measure $\rho$, defined by
\begin{align*}
\cZ_{n,p}(\gb;\rho) 
&:=\E\left[\frac{\gb}{n}\cdot \frac{\sum_{i<j}X_i X_j }{\norms{\mvX}^2_{n,p}}\right]\\
&= \int_{\dR^{n}} \exp\left(\frac{\gb n}{2}\cdot \frac{\abs{n^{-1}\sum_{i=1}^nx_i}^2}{\norms{\mvx}_{n,p}^2} - \frac{\gb}{2}\cdot \frac{n^{-1}\sum_{i=1}^{n}x_i^2}{\norms{\mvx}_{n,p}^2} \right) \prod_{i=1}^{n}d\rho(x_{i}).
\end{align*}

The above $\ell^{p}$ constraint model can be generalized with respect to any symmetric base measure $\rho$ on $\dR$ with $\P(X=0)=0$, $X\sim\rho$, and finite $p$-th moment under the self-scaled Hamiltonian
\begin{align*}
H_{n,p}(\mvx):= \frac{n}{2}\cdot \frac{\abs{n^{-1}\sum_{i=1}^{n}x_i}^2}{\norms{\mvx}_{n,p}^2} - \frac{1}{2}\cdot \frac{n^{-1}\sum_{i=1}^{n}x_i^2}{\norms{\mvx}_{n,p}^2}.
\end{align*}
\begin{ass}\label{assumption}
Let $X\sim\rho$ with $\E_\rho|X|^p=1$. In what follows, we assume one of the following:
\begin{enumerateA}
    \item\label{rhoass1} $\pr(|X|>0)=1$ and $\E_\rho e^{\theta |X|^p}<\infty$ for some $\theta>0$.
    \item\label{rhoass2} $\pr(|X|>0)=1$ and $\E_\rho[|X|^{-\ga}]<\infty$ for some $\ga>0$.
\end{enumerateA}
\end{ass}

Note that Assumption~\ref{rhoass1} is stronger than Assumption~\ref{rhoass2}~and is needed only for the large deviation argument (Theorem~\ref{thm:gen}). Using the GHS transform, we obtain a simpler variational form of the limiting free energy (Theorem~\ref{thm:genrho}) only under Assumption~\ref{rhoass2}. Since the Hamiltonian is normalized in the $\ell^p$ norm, it is necessary to assume that the base measure $\rho$ does not give a mass at 0 and decays fast enough near 0. Note that for $\rho_{p}$ distribution, Assumption~\ref{rhoass2} holds for any $0<\ga<1$. For Assumption~\ref{rhoass1}, one can find the same condition in~\cite{EN78}, where the generalized Curie--Weiss model was investigated.\\
\smallskip

\noindent\textbf{Large deviation based analysis.}\quad
We write $\cZ_{n,p}(\gb;\rho)$ to denote the corresponding partition function with $\gs_i,i\in[n]$ being i.i.d.~from $\rho$. That is, 
\begin{align*}
\cZ_{n,p}(\gb;\rho) = \int_{\dR^{n}} \exp\left(\frac{\gb n}{2}\cdot \frac{\abs{n^{-1}\sum_{i=1}^nx_i}^2}{\norms{\mvx}_{n,p}^2} - \frac{\gb}{2}\cdot \frac{n^{-1}\sum_{i=1}^{n}x_i^2}{\norms{\mvx}_{n,p}^2} \right) \prod_{i=1}^{n}d\rho(x_{i}).
\end{align*}
Since
\[
\frac{n^{-1}\sum_{i=1}^{n}x_i^2}{\norms{\mvx}_{n,p}^2}
\le
\begin{cases}
1, & p\ge 2,\\
n^{\tfrac{2}{p}-1}, & 1<p<2,
\end{cases} = o(n) \text{ for } p>1,
\]
as explained in Remark~\ref{rem:typical}, it suffices to consider the following partition function
\begin{align}\label{eq:ztildegen}
\widetilde{\cZ}_{n,p}(\gb;\rho) = \int_{\dR^{n}} \exp\left(\frac{\gb n}{2}\cdot \frac{\abs{n^{-1}\sum_{i=1}^nx_i}^2}{\norms{\mvx}_{n,p}^2}  \right) \prod_{i=1}^{n}d\rho(x_{i}).
\end{align}
Let $F(x,y) = \frac{\gb}{2}x^2y^{-2/p}$ and $\cX:=\{(x,y)\in\dR^2: 0\le |x|\le y^{1/p}\}$. Let $(X_i)_{i\ge 1}$ be an i.i.d.~sequence of random variables with density $\rho$, $S_n=\sum_{i=1}^n X_i$, and $T_n=\sum_{i=1}^n |X_i|^p$. Let $Q_n$ be the joint distribution of $(S_n/n, T_n/n)$. The log Laplacian and the Cram\'er transform of $(X,|X|^p)$ are given by
\begin{align*}
\psi_{\rho}(u,v):=\log\E_\rho e^{uX+v|X|^p},\qquad I(x,y):=\sup_{u,v} (ux+vy-\psi_{\rho}(u,v)).
\end{align*}
Then, the partition function $\widetilde{\cZ}_{n,p}(\gb;\rho)$ can be written as
\[
\widetilde{\cZ}_{n,p}(\gb;\rho) = \int_{\cX} e^{nF(x,y)}Q_n(dx,dy).
\]
Let $Q_{n,F}(A) = \widetilde{\cZ}_{n,p}(\gb;\rho)^{-1}\int_A e^{nF}dQ_n$.
We refer to the book~\cite{DZbook} for detailed discussions on large deviation principles. 

\begin{thm}\label{thm:gen}
Let $X\sim\rho$ with $\E_\rho|X|^p=1$ and $p\ge 1$ be fixed. Assume that Assumption~\ref{rhoass1} holds. Then, $\{Q_{n,F}\}_{n\ge 1}$ has a large deviation property with speed $n$ and a rate function $I_F(x,y):=I(x,y)-F(x,y)$ and
\begin{align}\label{eq:ld-plimit}
\lim_{n\to\infty}\frac1n\log\cZ_{n,p}(\gb;\rho) = \sup_{0\le |x|\le y^{1/p}}\left(F(x,y)-I(x,y)\right).        
\end{align}
Furthermore, if $K$ is a closed subset in $\dR^2$ that does not contain any minimum points of $I_F$, then $Q_{n, F}(K)\le e^{-C n}$ for large $n$ and some constant $C=C(K)>0$. 
\end{thm}

\noindent\textbf{Analysis using Generalized Hubbard--Stratonovich Transform.}\quad
For $X\sim\rho_p$, it is easy to check that
\begin{align*}
\sup_{0\le |x|\le y^{1/p}} &\left(\frac12\gb x^2y^{-2/p}-I(x,y)\right)\\
&=\sup_{0\le x\le y^{1/p}}\inf_{u,v}\left(\frac12\gb x^2y^{-2/p}+ \log\E_\rho e^{uX+v|X|^p}-ux-vy\right)\\
&\le \sup_{z,w\ge 0}\left( \log\E_\rho e^{\sqrt{\gb}zwX-|zX|^p/p} -\frac{p-2}{2p}\cdot w^{2p/(p-2)} \right)
\end{align*}
by taking 
\[
u = \sqrt{\gb}zw, v= -z^p/p \text{ where } x= w^{2p/(p-2)}/(\sqrt{\gb}zw), y= w^{2p/(p-2)}/z^{p}.
\]
It turns out that the equality holds under a weaker assumption on $\rho$. 

\begin{thm}\label{thm:genrho}
Let $X\sim\rho$ with $\E_\rho|X|^p=1$ and Assumption~\ref{rhoass2} holds. 
\begin{enumeratea}
\item If $p>2$, then 
\begin{align}\label{eq:htempvarp>2}
\lim_{n\to\infty}\frac1n\log\cZ_{n,p}(\gb;\rho) = \sup_{z,w\ge 0}\left\{ \psi_{\rho}(\sqrt{\gb}zw,-|z|^p/p) -\frac{p-2}{2p}\cdot w^{\frac{2p}{p-2}} \right\}.   
\end{align}
Furthermore, the limit is strictly positive if $\gb>\gb_c(\rho):=1/\E_\rho[X^2]$ and $0$ if $\gb\le 1$.
\item For $p=2$, we have
\[
\lim_{n\to\infty}\frac1n\log\cZ_{n,2}(\gb;\rho) = \sup_{z\ge 0}\psi_{\rho}(\sqrt{\gb}z, -z^2/2) .
\]
Furthermore, the limit is strictly positive for $\gb>\gb_c(\rho)$, and the limit is zero for $\gb<\gb_c(\rho)$.
\item If $1<p<2$, then 
\begin{align*}
F_p(\gb;\rho)
&:= \lim_{n\to\infty}\frac1n\log\cZ_{n,p}(\gb;\rho) \\
&=\sup_{y,c\ge 0}\inf_{t\ge 0 }\left\{ \frac{1}{2}\gb y^2+ \psi_{\rho}\left(ct, -ty/p\right)-\frac{p-1}{p}ty c^{p/(p-1)} \right\}
\end{align*}
and
\[
\sup_{z\ge 0 }\psi_{\rho}(\sqrt{y}z,-|z|^p/p)
=\sup_{x\ge 0}\left\{F_p(x y;\rho )-\frac{2-p}{2p}x^{\frac{p}{2-p}}\right\}.
\]
In particular, 
\begin{align*}
\lim_{n\to\infty}\frac1n\log\cZ_{n,p}(\gb;\rho)\le \inf_{w> 0 }\sup_{z\ge 0 }\left\{\psi_{\rho}\left({\sqrt{\gb}zw},-|z|^p/p\right)-\frac{p-2}{2p}w^{\frac{2p}{p-2}}\right\}.
\end{align*}
\end{enumeratea}
\end{thm}

Note that the representation~\eqref{eq:htempvarp>2} for $p>2$ case is a generalization of~\eqref{eq:classicHSvar} in the classical Curie--Weiss model. One can prove a Gaussian CLT for the scaled magnetization in the setting of Theorem~\ref{thm:genrho} for $\gb<1, p>2$ case. It is interesting to see if the maximizers in the right-hand side of~\eqref{eq:htempvarp>2} give us the information of the concentration of the magnetization as the large deviation results do in Theorem~\ref{thm:gen} for all $\gb<\gb_c(\rho)$. 

For $p\in(0,1)$, we generalize the results in Theorem~\ref{thm:rhop2} to the $\ell^p$ self-scaled model as follows.

\begin{thm}\label{thm:genrho2}
Let $(p_k)_{k\ge 1}$ and the function $\tau(\cdot)$ be as defined in Theorem~\ref{thm:rhop2}. Assume that for $X\sim\rho$, $\P(X=0)=0$, and one of the following holds:
\begin{enumeraten}
    \item\label{assum1} 
    $|\log\P(x\le X\le t x)|\le Cx^{\ga_1}$ 
    for some $C>0$, $t> 1$, and $\ga_1<2-p$, as $x\to \infty$.
    \item\label{assum2} $|\log\P(|X|<x)|\le C x^{-\ga_2}$ 
    for some $C>0$ and $\ga_2>2(1-p)$, as $x\to 0$.
\end{enumeraten}
For $p\in(0,1)$, we have
\[
\lim_{n\to\infty}n^{1-2/p}\log \cZ_{n,p}(\gb;\rho) = \frac{\gb}{2}\tau(p).
\]
\end{thm}

When $X\sim\rho_p$, we have $\P(|X|<x)\ge C x$ for small $x>0$ and so 
\[
|\log\P(|X|<x)|\le -C\log x \le Cx ^{-\ga_2}
\]
for any $\ga_2>2(1-p)$. Thus, $\rho_p$ satisfies Assumption~(\ref{assum2}).

\subsection{Heuristics behind the proof}\label{ssec:pfheu}

In this section, we provide some heuristics behind the analysis of the $\ell^p$ constraint model. A similar analysis can also be used for the $\ell^p$ self-scaled model. Let $\mvX=(X_1,X_2,\ldots,X_n)$ be distributed with density proportional to 
\[
\exp\left(\frac{\gb n}{2}\cdot \frac{\abs{n^{-1}\sum_{i=1}^nx_i}^2}{\norms{\mvx}_{n,p}^2} - \frac{\gb}{2}\cdot \frac{n^{-1}\sum_{i=1}^{n}x_i^2}{\norms{\mvx}_{n,p}^2} \right) \prod_{i=1}^{n}\rho_{p}(x_{i}).
\]
One can easily check that the conditional density of $X_i$ given $\{(X_j)_{j\neq i}=(x_j)_{j\neq i}\}$ at $x$ is proportional to
\[
\exp\left(\gb\cdot \frac{n^{-1}\sum_{j\neq i}x_j}{\norms{\mvx}_{n,p}^2}\cdot x -\frac{|x|^p}{p}- O(1/n) \right),\qquad  x\in\dR.
\]
Due to the mean-field structure, if the following concentration behavior (unproven) is valid for large $n$:
\[
\frac1n\sum_{j=1}^n X_i \approx a_1,\text{ and } 
\frac1n\sum_{j=1}^n |X_i|^p \approx a_p^p,
\]
then we should have $a_1 = \E Y \text{ and } a_p^p=\E|Y|^p,$ where $Y$ has density proportional to
\[
\exp(\gb\cdot a_1/a_p^2\cdot y - |y|^p/p), y\in\dR.
\]
Defining, $\theta=\gb\cdot a_1/a_p^2$, we must have (see Lemma~\ref{lem:psiibp})
\[
a_1=\E[Y] = \psi_p'(\theta),\text{ and } a_p^p=\E[|Y|^p] = 1+\theta \psi'_p(\theta)
\]
where $\psi_p(t)=\log \E \exp(tZ_p)$. Thus, $\theta$ should satisfy the equation
\begin{align}\label{eq:theta}
\theta &= \gb\cdot \frac{\psi_p'(\theta)}{(1+\theta \psi'_p(\theta))^{2/p}}.
\end{align}
Lemma~\ref{lem:psidec} shows that  the function
\[
t\mapsto 
\frac{t}{\psi_p'(t)} \cdot  (1+t \psi'_p(t))^{2/p},\qquad t>0
\]
is strictly increasing for $p\ge 2$.
Now, $\theta=0$ is always a solution to~\eqref{eq:theta} and there is a unique positive solution if $p\ge 2$ and
\begin{align*}
\gb > \hat\gb_c:=\inf_{\theta\ge 0} \theta \cdot  {(1+\theta \psi'_p(\theta))^{2/p}} / \psi_p'(\theta).
\end{align*}
Thus, for $\gb\le \hat\gb_c$, we expect the magnetization to be concentrated at $0$ and a typical coordinate $X_i$ to be distributed according to $\rho_p$. Whereas, for $\gb>\hat\gb_c$, the magnetization should not be concentrated at zero. For $p\ge 2, \gb>\gb_c$, this can be seen from the following computation.  For $\theta>0$ satisfying~\eqref{eq:theta} and with $a:= \psi_p'(\theta), b:=1+\theta \psi'_p(\theta)$, we have
\begin{align*}
\cZ_{n,p}(\gb)
&\approx  
\int_{\dR^{n}} b^{-n/p} \exp\left(\frac{\gb n}{2}\cdot \frac{\abs{n^{-1}\sum_{i=1}^nx_i}^2}{\norms{\mvx}_{n,p}^2} - \frac1{pb}\sum_{i=1}^n|x_i|^p \right)c_p^n\prod_{i=1}^{n}dx_i\\
&\ge \exp\left(n\psi_p(\theta)- \frac{n}2\theta a +\frac{n}{p}(b-1) -\frac{n}{p}\log b\right)\cdot  \\
&\qquad\int_{\sum_{i=1}^nx_i \approx na+O(\sqrt{n}), \sum_{i=1}^n|x_i|^p \approx nb+O(\sqrt{n})}  \prod_{i=1}^{n} c_p e^{\theta x_i-|x_i|^p/p - \psi_p(\theta)}dx_i.
\end{align*}
Note that the last integral is $\Theta(1)$ by CLT and by Lemma~\ref{lem:psidec}, we have
\begin{align*}
\psi_p(\theta)&- \frac12\theta a +\frac{\theta}{p} (b-1) -\frac{1}{p}\log b\\
&= \psi_p(\theta) -\frac12\theta \psi'_p(\theta) +\frac1p \theta \psi'_p(\theta) - \frac1p\log(1+\theta \psi'_p(\theta))\\
&\ge \psi_p(\theta) -\frac12\theta \psi_p'(\theta) 
= \int_0^\theta t \cdot \left(\frac{\psi_p'(t)}{t} - \frac{\psi_p'(\theta)}{\theta}\right)dt >0.
\end{align*}
The generalized Hubbard--Stratonovich transform is used to understand the distribution of the typical coordinates in a structured way and make the above heuristics rigorous.

\subsection{Discussion on the Critical Temperatures}\label{ssec:critical}

We have seen that the $\ell^p$ constraint Curie--Weiss model for $p>2$ has the critical inverse temperature $\gb_c(p)$ as in~\eqref{eq:betacp} in a sense that if $\gb<\gb_c(p)$ then the limiting free energy
\[
\lim_{n\to\infty}\frac1n \log\cZ_{n,p}(\gb)=0
\]
and for $\gb>\gb_c(p)$ then the limiting free energy is strictly positive.

In the $\ell^p$ self-scaled model with general base measure $\rho$, Theorem~\ref{thm:genrho} (a) tells us that the limiting free energy is strictly positive when $\gb>\gb_c(\rho):=1/\var_\rho(X)$. However, it is unclear that $\gb_c(\rho)$ is indeed the critical inverse temperature. To show this, it follows from~\eqref{eq:htempvarp>2} that we need to show 
\begin{align}\label{eq:mgfineq}
\E \exp\left(\sqrt{\gb}zwX-z^p |X|^p/p\right)\le \exp\left(\frac{p-2}{2p}w^{2p/(p-2)}\right)
\end{align}
for all $z,w\ge 0, \gb\le \gb_c(\rho)$, where $X\sim\rho$. 

One can derive the inequality for $\rho=\rho_p$, as mentioned in the discussion after Theorem~\ref{thm:rhop}. The inequality~\eqref{eq:mgfineq} holds true for $\rho=\frac12\gd_{1}+\frac12\gd_{-1}$, which boils down to the classical Curie-Weiss model. Indeed,
\begin{align*}
\psi_\rho(\sqrt{\gb_c(\rho)}zw,-z^p/p)
&=\E_\rho\exp\left(\sqrt{\gb_c(\rho)}zwX-z^p |X|^p/p\right)\\
&=\cosh(zw)e^{-z^p/p}
\le\exp\left(\frac{z^2 w^2}{2}-\frac{z^p}{p}\right)
\le \exp\left(\frac{p-2}{2p}w^{2p/(p-2)}\right).
\end{align*}
Here, we used the bound $\cosh(t)\le \exp(t^2/2)$ and Young's inequality $ab\le a^q/q + b^r/r$ for $q=p/2$ and $1/q+1/r=1$. The inequality~\eqref{eq:mgfineq} remains open for general symmetric measures. Note that one can show~\eqref{eq:mgfineq} for $\gb\le 1$ using Young's inequality.

For $p=2$ case, Theorem~\ref{thm:genrho} (b) shows that $\gb_c(\rho)$ is the critical inverse temperature for the $\ell^p$ self-scaled model. In the case where $1<p<2$, it is not clear whether $\gb_c(p)$ is the critical inverse temperature even for the $\rho_p$ case. 

For $\rho=\rho_1$, one can show that $\gb_c(1)=\frac12$ is the critical inverse temperature. For $X\sim\rho_1$ and $0\le x< y$, we have
\begin{align*}
I(x,y)&=\sup_{u,v}(ux+vy-\log\E\exp(uX+v|X|))\\
& =\sup_{0\le u < w}(ux+(1-w)y -\log w+\log(w^{2}-u^{2}))\\
& =y-\log y+\sup_{0\le s < 1, z> 0}(zst-z +\log z+\log(1-s^2))\\
& =y-1-\log y+\sup_{0\le s < 1}(\log(1-s^2)-\log(1-st))\\
&= y-1-\log y -\log\frac12(1+\sqrt{1-t^2})
\end{align*}
where $t:=x/y$ and the third equality follows by a change of variable $t=u/w, z=wy$. In particular, we get
\begin{align*}
\sup_{0\le x\le y}(\gb x^{2}/2y^{2} - I(x,y)) = \sup_{0\le t\le 1}\left(\frac12\gb t^2 +\log\frac12(1+\sqrt{1-t^2})\right).
\end{align*}
By changing the variable $t=\sin\theta$, one can quickly check that the optimizer is attained at $t_*$ where
\[
t_*^2:=\frac{2(2-1/\gb)_+}{\sqrt{\gb^2+4\gb}+2-\gb}.
\]
If $\gb<\gb_c(1)=\frac12$, then $t_*=0$ and the limiting free energy is zero. On the other hand, if $\gb>\gb_c(1)=\frac12$, then one can see that the limiting free energy is strictly positive.

\subsection{Related Literature}

If we remove the $\ell^p$ scaling in the partition function~\eqref{eq:Zrhop}, our model boils down to the generalized Curie--Weiss model in~\cite{EN78}, where the base measures are given by the product measure with sub-Gaussian tail decay. From this perspective, our model generalizes to the one in~\cite{EN78} by introducing an extra dependence structure on the spin configuration, which arises from the $\ell^p$ constraint. Under Gaussian tail decay assumption for $\rho$, Ellis and Newman~\cite{EN78} proved that the Curie--Weiss model without constraint exhibits the phase transition at $\gb_c(\rho) = 1/\var_\rho(X)$. It was shown in~\cite{EiseleEllis} that the generalized Curie--Weiss model can have a sequence of phase transitions at different critical temperatures.

Since the self-organized criticality model in~\cite{CG16} is the special case where $p=2$ and $\gb=\gb_c(2)$ is the critical inverse temperature, our model can also be considered as a generalization of the self-organized criticality model. Detailed analysis at criticality $\gb_c(2)$ was investigated in~\cite{G14, GorVar15}. 

We simplified our model with self-scaled random vectors with $\rho_p$ distribution and extended to a general $\ell^p$ self-scaled Curie--Weiss model. From this perspective, our model is closely connected to self-normalized processes in~\cite{PLS}. 

Corresponding disordered model is known as $\ell^p$–Grothendieck spin–glass model~\cite{ChenSen23, Domin22, KhoNao12}. For an $n\times n$ matrix $A$, it is well-known that the maximum of $\langle Ax, x\rangle$ for $\|x\|_2\le 1$ is the largest eigenvalue of $\frac12 (A+A^T)$. A natural question is the optimization of $\langle Ax, x\rangle$ for $\|x\|_p\le 1$ and $1\le p \le \infty$. This is called the $\ell^p$-Grothendieck problem. The same optimization problem with random matrix $A$, where the entries are i.i.d.~standard Gaussian, has been studied in~\cite{ChenSen23, Domin22,KhoNao12}. The limiting free energy is given in~\cite{ChenSen23}.

The distribution of a typical point from the intersection of high-dimensional symmetric convex bodies shows phase transition. \cite{SChatt17} considered the uniform probability distribution on the intersection of $\ell^1$-ball and $\ell^2$-ball and showed that phase transitions and localization phenomena occur depending on the radius of $\ell^2$-ball. Suppose that $X=(X_1, X_2,\ldots,X_n)$ is a random vector uniformly chosen from $\{x\in\dR^n: \sum_{i=1}^n|x_i|=n,\sum_{i=1}^n|x_i|^2=bn \}$ for $b>1$. If $1<b\le 2$, then the first $k$ random variables in $(X_1, X_2, \ldots, X_n)$ converges to an i.i.d.~sequence of some random variables depending on $b$ in law and in joint moments as $n\to\infty$. On the other hand, if $b>2$, then $(X_1, X_2, \ldots, X_k)$ converges to an i.i.d~sequence of exponential random variables, and the localization phenomenon appears in a sense that most of the mass is concentrated on the maximum component of $(X_1, X_2, \ldots, X_n)$ and the other components have negligible masses as $n\to\infty$.

The result of~\cite{SChatt17} has been extended to multiple constraints by~\cite{Nam20}. He considered the microcanonical ensembles with multiple constraints with unbounded observables and showed that a sequence $(X_1, X_2,\ldots, X_n)$ sampled from the multiple constraints exhibits a phase transition, localization, and delocalization phenomena. Note that the microcanonical ensemble with unbounded multiple constraints has been of great interest due to its connection to nonlinear Schr\"{o}dinger equations.   

\subsection{The Classical Curie--Weiss Model}\label{ssec:classicCW}

Curie--Weiss model or Ising model on the complete graph is one of the simplest spin models exhibiting phase transition with respect to the temperature in terms of magnetization or average spin. The spins are arranged on the vertices of the complete graph $K_{n}=([n],\cE_{n})$ on $n$-vertices. The spin at vertex $i\in[n]$ is denoted by $\gs_{i}$ and the Hamiltonian corresponding to a spin configuration $\mvgs=(\gs_{i})_{i=1}^{n}$ is given by 
\begin{align*}
H_{n}(\mvgs)= \frac{1}{n}\sum_{i<j}\gs_{i}\gs_{j} = \frac{n}{2}\left(\frac1n\sum_{i=1}^{n}\gs_{i} \right)^{2} - \frac{1}{2} \quad\text{ for } \mvgs\in S_{n,\infty}:=\{-1,+1\}^{n}.
\end{align*}
The Gibbs measure on $\{-1,+1\}^{n}$ at inverse temperature $\gb>0$ is defined as 
\begin{align*}
d\mu_{n,\gb}(\mvgs) = \cZ_{n}(\gb)^{-1} \exp(\gb H_{n}(\mvgs)) d\mvgs
\end{align*}
where $d\mvgs$ is the uniform measure on  $\{-1,+1\}^{n}$ and $\cZ_{n}(\gb)$ is the partition function. Several approaches are available in the literature to analyze phase transition in the Curie--Weiss model. Next, we describe three of the approaches. 

\subsubsection{Combinatorial approach} 

One can prove the phase transition of the magnetization using direct combinatorial arguments, equivalent to large deviation analysis. Indeed, it follows from an elementary counting and Stirling's formula that for any closed subset $K\subset [-1,1]$, 
\begin{align*}
\frac1n \log \mu_{n,\gb}(\{m(\mvgs)\in K\}) = 
-\min_{x\in K}\left\{(I(x)-F(x))-\min_{y\in[-1,1]}(I(y)-F(y))\right\}
\end{align*}
where $I(x)=\left(\tfrac{1+x}{2}\right)\log\left(\tfrac{1+x}{2}\right)+\left(\tfrac{1-x}{2}\right)\log\left(\tfrac{1-x}{2}\right)$ and $F(x)=\tfrac{\beta x^2}{2}$. Thus, the magnetization is concentrated on the minimum points of $I-F$, and the minimizer satisfies $x=\tanh(\gb x)$. If $\beta<1$, then the equation has the unique solution $x=0$ so that the magnetization is concentrated at $0$. On the other hand, for $\beta>1$,  the magnetization is concentrated at $\pm x_{\gb}^{\ast}$, which are nonzero solutions to $x=\tanh(\gb x)$.

\subsubsection{Exchangeable pair approach} 

Another approach is to use the Glauber dynamics. Here, we describe Stein's method for exchangeable pairs using Glauber dynamics. Suppose $\mvgs$ is sampled from $\mu_{n,\gb}$ and $I\in [n]$ is chosen at random. We define $\mvgs'$ as follows. If $I=j$, then $\gs_i':=\gs_i$ for $i\neq j$ and $\gs_j'$ is sampled from $\mu_{n,\gb}$ conditioned to $(\gs_i)_{i\neq j}$. Then, $(\mvgs, \mvgs')$ is an exchangeable pair. For $W=n^{-1/2}\sum_{i=1}^{n}\gs_i$ and $W'=n^{-1/2}\sum_{i=1}^{n}\gs_i'$, we get $\E[W-W'\mid W]\approx n^{-1}(1-\gb)W$ for $\gb<1$ and $\E[W-W'\mid W]\approx n^{-2}W^3/3$ for $\gb=1$. Classical Stein's method argument shows that $W$ converges to normal for $\gb<1$, and it converges to a non-normal distribution for $\gb=1$ under suitable scaling (see~\cite{CS11}). We also note that~\cite{Chatt07} used this exchangeable pair approach to obtain concentration inequalities for the magnetization. 

\subsubsection{Hubbard--Stratonovich transform} 

One can study the phase transition of the Curie--Weiss model using analytic methods. Hubbard--Stratonovich transform is
\begin{align*}
e^{{x^2}/{2}} = \frac{1}{\sqrt{2\pi}}\int_{\dR}\exp(-{y^2}/{2}+xy)\, dy.
\end{align*}
Since $H_{n}(\mvgs)= \tfrac{1}{2n}(\sum_{i=1}^{n}\gs_i)^2-\tfrac{1}{2n}\sum_{i=1}^{n}\gs_i^2= \tfrac{1}{2n}(\sum_{i=1}^{n}\gs_i)^2-\tfrac{1}{2}$, we have
\begin{align*}
\cZ_{n}(\gb)
= \E[e^{\gb H_n(\mvgs)}]
&= \frac{e^{-{\gb}/{2}}}{\sqrt{2\pi}}\int_{\dR}\exp\left(-y^2/{2}+n\log \cosh\left(\sqrt{\gb/n}\cdot y\right)\right)\, dy\\
&=\frac{e^{-{\gb}/{2}}\sqrt{n}}{\sqrt{2\pi}}\int_{\dR}\exp\left(n\left(-{y^2}/{2}+\log \cosh\left(\sqrt{\gb} \cdot y\right)\right)\right)\, dy.
\end{align*}
Thus,
\begin{align}\label{eq:classicHSvar}
\lim_{n\to\infty}\frac1n \log \cZ_n(\gb)
= \sup_{y}\left\{-{y^2}/{2}+\log \cosh\left(\sqrt{\gb} y\right)\right\}.
\end{align}
If $\gb<1$, $W=\frac{1}{\sqrt{n}}\sum_{i=1}^{n}\gs_i$, and $t>0$, it follows from Hubbard--Stratonovich transform and the dominated convergence theorem that
\begin{align*}
\lim_{n\to\infty}\E_{\mu_{n,\gb}}[e^{tW}]
&=\lim_{n\to\infty}\frac{\int_{\dR}\exp\left(-{y^2}/{2}+n\log \cosh\left(n^{-1/2}(\sqrt{\gb}\cdot y-t)\right)\right)\, dy }{\int_{\dR}\exp\left(-{y^2}/{2}+n\log \cosh\left(\sqrt{\gb/n} \cdot y\right)\right)\, dy }\\
&=\exp\left({t^2}/{2(1-\gb)}\right).
\end{align*}
This shows that $W$ converges to a normal distribution as $n\to\infty$. For further details on the classical Curie--Weiss model, we refer the reader to~\cite[Chap. 2]{FV18}.

In this paper, we mainly focus on developing and applying the generalized Hubbard--Stratonovich (GHS) transform to our model and using the classical large deviation tools. Appropriate stationary dynamics and application of Stein's method will be discussed in a future article.

\subsection{Notation}

The indicator function for an event $A$ is denoted by $\ind_A$.
For random variables $X,Y$, we use the notations $X\equald Y$ when they have the same distribution. If $X$ and $Y$ are independent, we write $X\perp Y$. If the distribution of a random variable $X$ is given by its density or probability measure $\rho$, we denote by $X\sim\rho$. The normal distribution with mean $\mu$ and variance $\gs^2$ is denoted by $\mathrm{N}(\mu,\gs^2)$. 
The expectation with respect to a random variable $X\sim\rho$ is denoted by $\E_X=\E_\rho$ and the subscript can be omitted when there is no ambiguity.
For a probability measure $\nu$, we use the notation $\langle f \rangle_\nu = \E[f(X)]$ where $X\sim \nu$. For an $n$-dimensional vector $\mvx =(x_1,x_2,\ldots,x_n)$ and $p>0$, the normalized $\ell^p$ norm is denoted by
\[
\norms{\mvx}_{n,p}
=\Big(\frac1n \sum_{i=1}^n |x_i|^p\Big)^{1/p}.
\]

\subsection{Organization}

The paper is organized as follows. The model is defined in Sections~\ref{ssec:model} and~\ref{ssec:gen}. The main results are presented in Section~\ref{ssec:mainres}.  Section~\ref{sec:auxlem} contains auxiliary lemmas that will be frequently used in the proofs of main results.  Section~\ref{sec:pf-free} contains the proofs of Theorems~\ref{thm:rhop},~\ref{thm:typical} and Proposition~\ref{prop:selfnormal} about the behavior of the $\ell^{p}$ constraint Curie--Weiss model.  In Section~\ref{sec:GHSproof}, we prove Theorem~\ref{thm:ghst} and Theorem~\ref{thm:lgamma} for the generalized Hubbard--Stratonovich transform.  Section~\ref{sec:LargeDev} contains proofs of Theorems~\ref{thm:gen},~\ref{thm:genrho}, and~\ref{thm:genrho2} regarding the generalization to $\ell^p$ self-scaled model. Finally, Sections~\ref{sec:mscw} and~\ref{sec:oq} contain discussions about the multi-species model and open questions, respectively.

\section{Auxiliary Lemmas }\label{sec:auxlem}

\subsection{Cumulant generating function}\label{ssec:cgflem}

Let $0<p<\infty$ and $Z_p$ be a random variable with density $\rho_p(x) = c_p e^{-|x|^p/p}$. Note that $\E[Z_p^2]=\nu_p^2 = 1/\gb_c$ and $\E[|Z_p|^p]=1$. Recall the cumulant-generating function 
\[
\psi_p(t) = \log \E[e^{tZ_p}].
\]

\begin{lem}\label{lem:mgfbd}
For all real number $\ell \ge 2$ we have
\[
\nu_p^{-\ell}\E \abs{Z_p}^{\ell} \text{ is decreasing in } p>0.
\]
In particular, for all $t\in\dR, p\ge 2$ we have $\psi_p(t)\le \frac12\nu_p^2t^2$ and $\lim_{t\to 0}\psi_p(t)/t^2 = \frac12\nu_p^2$.
\end{lem}
\begin{proof}
Fix $\ell\ge2$ and $p\ge q>0$. To prove that $\nu_p^{-\ell}\E \abs{Z_p}^{\ell} \le \nu_q^{-\ell}\E \abs{Z_q}^{\ell}$ or equivalently, $\E \abs{Z_q}^{\ell} /\E \abs{Z_p}^{\ell} \ge (\E Z_q^2 /\E Z_p^2)^{\ell/2}$, we use Proposition~\ref{prop:decomp}. We have $Z_q\equald Z_p\cdot U_{p,q}$ where $U_{p,q}$ is independent of $Z_p$. Thus, 
\[
\E \abs{Z_q}^{\ell} /\E \abs{Z_p}^{\ell} = \E U_{p,q}^\ell\ge (\E U_{p,q}^2)^{\ell/2} = (\E Z_q^2 /\E Z_p^2)^{\ell/2}.
\]
Using symmetry of $Z_p$, we have, for all $t\in\dR$,
\[
\E \exp(tZ_p) = \sum_{k=0}^\infty \frac{t^{2k}}{(2k)!}\E Z_p^{2k}.
\]
Combining, we get for $p\ge 2$,
\[
\E \exp(tZ_p) \le \E \exp(t\nu_p Z_2)=\exp(\nu_p^2t^2/2).
\]
This completes the proof.
\end{proof}

\begin{lem}\label{lem:psidec}
The function $\psi_p'(t)/t$ is decreasing in $t\in[0,\infty)$ for $p\ge 2$, and increasing in $t\in[0,\infty)$ for $0<p<2$. 
\end{lem}

\begin{cor}
    Let $X$ be a r.v. with density proportional to $\exp(x - c|x|^p), x\in\dR$ for some $c>0, p\ge 2$. Then, we have
    $
    \var(X)\le \E X.
    $
\end{cor}

\begin{proof}[Proof of Lemma~\ref{lem:psidec}]
If $p=2$, then $\psi_2'(t)/t=\E Z_2^2$ is constant and the result holds. For the rest of the proof, we fix $p>0$.

Since $|Z_p|^p/p$ is Gamma distributed with parameters $1/p$ and $1$, we get that $\E Z_p^{2k}=p^{(2k+1)/p}\Gamma((2k+1)/p)/\Gamma(1/p)$ for all $k>0$. In particular, we have
\begin{align*}
    e^{\psi_p(t)}=\E e^{tZ_p}&= \frac{p^{1/p}}{\Gamma(1/p)}\cdot \sum_{k=0}^\infty \frac{t^{2k}p^{2k/p}}{(2k)!}\cdot \Gamma((2k+1)/p),\qquad t>0.
\end{align*}
Thus, it is enough to show that the function $f(t):=\log M'(t)-\log (t M(t))$ is decreasing (resp.~increasing) for $p>2$ (resp.~$p<2$), where
\begin{align*}
M(t) &:= \sum_{k=0}^\infty \frac{t^{2k}}{(2k)!}\cdot \Gamma((2k+1)/p),\qquad t>0.
\end{align*}
We have, 
\begin{align*}
    tM(t) &= p\cdot \sum_{k=0}^\infty \frac{t^{2k+1}}{(2k+1)!}\cdot \Gamma((2k+1+p)/p) = p\cdot \sum_{k=0}^\infty \frac{t^{2k+1}}{(2k+1)!}\cdot a_p(k),\\
    M'(t) &=  \sum_{k=0}^\infty \frac{t^{2k+1}}{(2k+1)!}\cdot \Gamma((2k+3)/p) = \sum_{k=0}^\infty \frac{t^{2k+1}}{(2k+1)!}\cdot a_2(k).
\end{align*}
where $a_i(k):=\Gamma((2k+1+i)/p)$ for $i>0$. Thus, $f'(t) \le 0$ iff $M''(t)\cdot tM(t) - M'(t)\cdot (tM(t))'\le 0.$ Moreover, we have
\begin{align*}
    &\frac1p\cdot \left[M''(t)\cdot tM(t) - (tM(t))'\cdot M'(t)\right] \\
    &=\sum_{k=0}^\infty \frac{t^{2k}}{(2k)!}\cdot a_2(k) \cdot \sum_{\ell=0}^\infty \frac{t^{2\ell+1}}{(2\ell+1)!}\cdot a_p(\ell)-\sum_{k=0}^\infty \frac{t^{2k}}{(2k)!}\cdot a_p(k) \cdot \sum_{\ell=0}^\infty \frac{t^{2\ell+1}}{(2\ell+1)!}\cdot a_2(\ell).
\end{align*}
For $n\ge 0$, the coefficient of $t^{2n+1}$ in the left hand side is
\begin{align*}
&\sum_{k=0}^n
\frac{1}{(2k)!(2(n-k)+1)!}
\biggl(a_2(k)a_p(n-k)-a_p(k)a_2(n-k)\biggr)\\
&\qquad=  
\frac{1}{2}
\Bigg(
\sum_{k=0}^n
\frac{2k+1}{(2k+1)!(2(n-k)+1)!}
(a_2(k)a_p(n-k)-a_p(k)a_2(n-k))\\
&\qquad\qquad+\sum_{k=0}^n
\frac{2(n-k)+1}{(2k+1)!(2(n-k)+1)!}
(a_2(n-k)a_p(k)-a_p(n-k)a_2(k))
\Bigg)\\
&\qquad =
\sum_{k=0}^n
\frac{(k-(n-k))a_2(n-k)a_2(k)}{(2k+1)!(2(n-k)+1)!}\left(\frac{a_p(k)}{a_2(k)}-\frac{a_p(n-k)}{a_2(n-k)}\right).
\end{align*}
Since the Gamma function is log-convex, one can check that the function
\[
x\mapsto \frac{\Gamma(x+1)}{\Gamma(x+2/p)}
\]
is decreasing in $x$ for $p\ge 2$, and increasing in $x$ for $p<2$.
Since $\frac{a_p(k)}{a_2(k)}=\frac{\Gamma(x+1)}{\Gamma(x+2/p)}$ for $x=(2k+1)/p$, we obtain that 
\[
(k-(n-k))\left(\frac{a_p(k)}{a_2(k)}-\frac{a_p(n-k)}{a_2(n-k)}\right)
\]
is negative if $p\ge 2$ and otherwise positive, which completes the proof.
\end{proof}

\begin{lem}\label{lem:psiibp}
For all $t>0$,
\[
\frac{\E[|Z_p|^p e^{tZ_p}]}{\E[ e^{tZ_p}]}=1+t\psi_p'(t).
\]
\end{lem}
\begin{proof}
Using integration by parts, we have
\begin{align*}
\E[e^{t Z_p}]
&=c_p\int_{\R}e^{tx-|x|^p/p}\, dx\\
&=-c_p\int_{\R}(tx-|x|^p)e^{tx-|x|^p/p}\, dx\\
&=\E[|Z_p|^p e^{t Z_p}]-t\E[Z_p e^{t Z_p}]
=\E[|Z_p|^p e^{t Z_p}]-t\E[e^{t Z_p}]\psi_p'(t),
\end{align*}
which completes the proof.
\end{proof}

\subsection{Typical configuration}\label{ssec:typlem}

Let $p> 2$. Let $X_1,X_2,\ldots,X_n$ be i.i.d.~with density $\rho_p$, $S_n = \sum_{i=1}^n X_i$, and $T_n = \sum_{i=1}^n |X_i|^p$. Since $T_n$ is independent of $S_n/T_n^{1/p}$, we have
\begin{align*}
\widehat{\cZ}_{n,p}(\gb)
&:= \E\left[\bigl(T_{n}/n\bigr)^{-1/p} \exp\left(\frac{\gb}{2n}\cdot  S_n^2\cdot \bigl(T_n/n\bigr)^{-2/p}\right)\right]\\
&= \E\left[\bigl(T_{n}/n\bigr)^{-1/p}\right] \E\left[\exp\left(\frac{\gb}{2n}\cdot  S_n^2\cdot \bigl(T_n/n\bigr)^{-2/p}\right)\right].
\end{align*}
Let $U=U_{p,2}$ be independent of $X$ such that $X\cdot U\sim\N(0,1)$ and $\theta$ be the density of $U$. It follows from Theorem~\ref{thm:ghst} that
\begin{align*}
\widehat{\cZ}_{n,p}(\gb) 
= c_{p}\cdot\E_{U} \int_\dR \E\left[ \exp\left(\sqrt{\frac{\gb}{n}}\cdot S_n z U - \frac1{np}T_n\abs{z}^p\right)\right] \,dz.
\end{align*}
Define the joint density of $(X_1, X_2, \ldots, X_n, Z^{(n)}, U^{(n)})$ by
\begin{align*}
f(x_1,x_2,\ldots,x_n,z,u) 
= \frac{c_p}{\wh{\cZ}_{n,p}(\gb)} \exp\left(\sqrt{\frac{\gb}{n}}\cdot  z u\sum_{i=1}^n x_i - \frac1{np}\abs{z}^p\sum_{i=1}^n |x_i|^p\right)\theta(u)\, \prod_{i=1}^n \rho_p(x_i).
\end{align*}

\begin{lem}\label{lem:condden}
The random variables $X_1,X_2,\ldots,X_n$ are conditionally i.i.d.~given $Z^{(n)}=z, U^{(n)}=u$ for $z\in \dR$ and $u> 0$ with the conditional density for $X_1$
\begin{align}\label{eq:condden}
&f(x_1\mid z,u)\\
&\quad=\exp\left(\sqrt{\frac{\gb}{n}}zu x_1 - \frac1p (1+\frac{|z|^p}{n})|x_1|^p-\psi_p\left(\sqrt{\frac{\gb}{n}}\frac{zu}{(1+\frac{|z|^p}{n})^{1/p}}\right)\right)(1+{|z|^p}/{n})^{1/p}\, dx_1.\notag
\end{align}
Furthermore, we have
\begin{align*}
\E\left[X_1\mid Z^{(n)}=z, U^{(n)}=u\right]
&=\frac{1}{\sqrt{n}}\left(\sqrt{\frac{\gb}{\gb_c^2}} \cdot zu + o(1)\right),\text{ and }
\\
\var(X_1\mid Z^{(n)}=z, U^{(n)}=u)
&=\frac{1}{\gb_c}+o(1).
\end{align*}
\end{lem}
\begin{proof}
Since the density of $(Z^{(n)}, U^{(n)})$ is given by
\begin{align*}
f_{Z^{(n)}, U^{(n)}}(z,u) 
= \frac{c_p}{\wh{\cZ}_{n,p}(\gb)} \exp\left( n\left( \psi_p\left(\sqrt{\frac{\gb}{n}}\frac{zu}{(1+\frac{|z|^p}{n})^{1/p}}\right)-\frac1p\log\left(1+\frac{|z|^p}{n}\right) \right) \right) \theta(u),
\end{align*}
it is easy to see that $X_1,X_2,\ldots,X_n$ are conditionally i.i.d. given $Z^{(n)}=z, U^{(n)}=u$ with the common density~\eqref{eq:condden}. By change of variables, 
\begin{align*}
\E[X_1\mid Z^{(n)}=z, U^{(n)}=u]
&= \left(1+\frac{|z|^p}{n}\right)^{-1/p}\int_{\dR}x\exp(\theta x-|x|^p/p-\psi_p(\theta))\, dx\\
&= \left(1+\frac{|z|^p}{n}\right)^{-1/p}\psi_p'(\theta),\\
\theta
&=\sqrt{\frac{\gb}{n}}\frac{zu}{(1+\frac{|z|^p}{n})^{1/p}}.
\end{align*}
Since $\theta\to 0 $ as $n\to\infty$ and $\lim_{\theta\to 0 }\psi_p'(\theta)/\theta = \nu_p^2 = 1/\gb_c$, we get
\begin{align*}
\E[X_1\mid Z^{(n)}=z, U^{(n)}=u]
&=\left(1+\frac{|z|^p}{n}\right)^{-1/p}\psi_p'(\theta)
=\frac{1}{\sqrt{n}}\left(\sqrt{\frac{\gb}{\gb_c^2}} \cdot zu + o(1)\right).
\end{align*}
Similarly, since $\psi_p''(\theta)\to \nu_p^2$ as $\theta\to 0$, 
\begin{align*}
\var(X_1\mid Z^{(n)}=z, U^{(n)}=u)
&=\left(1+\frac{|z|^p}{n}\right)^{-2/p}\psi_p''(\theta)\\
&= \left(1+\frac{|z|^p}{n}\right)^{-2/p} \nu_p^2\cdot (1+o(1))
=\frac{1}{\gb_c}+o(1).
\end{align*}
This completes the proof.
\end{proof}

\section{Proofs for the $\ell^{p}$ constraint Curie--Weiss model}\label{sec:pf-free}

In this section, we provide the proofs of Theorems~\ref{thm:rhop},~\ref{thm:typical}, and Proposition~\ref{prop:selfnormal}, where the limiting free energy, Gaussian fluctuation, and the typical configuration of the $\ell^p$ constraint Curie--Weiss model are discussed.

\subsection{Proof of Theorem~\ref{thm:rhop} (a)}

First, we show that $\cZ_{n,p}(\gb)$ is uniformly bounded in $n$. Fix $p\ge 2$. Let $X,X_{1},X_{2},\ldots,X_{n}$ be i.i.d.~random variables from $\rho_{p}$. Define the random variables 
\[
S_{n}:=\sum_{i=1}^{n}X_{i}\text{ and } T_{n}:=\sum_{i=1}^{n}\abs{X_{i}}^{p}.
\]
Now, we define
\begin{align*}
\widehat{\cZ}_{n,p}(\gb)
&:= \E\left[\bigl(T_{n}/n\bigr)^{-1/p} \exp\left(\frac{\gb}{2n}\cdot  S_n^2\cdot \bigl(T_n/n\bigr)^{-2/p}\right)\right].
\end{align*}
Since $T_n/p$ is a Gamma random variable with shape parameter $n/p$ and the rate parameter 1, we have
\begin{align*}
\widehat{\cZ}_{n,p}(0) = n^{1/p}\E [T_{n}^{-1/p}] = \frac{(n/p)^{1/p}\cdot \gC((n-1)/p)}{\gC(n/p)}  \le (n/(n-p))^{1/p} \le \exp(1/(n-p)),
\end{align*}
where the first inequality follows by convexity of the log-Gamma function (also known as Gautschi's inequality) and thus, is uniformly bounded for $n>p$.

Let $\gb>0$ and $U\equald U_{p,2}$ independent of $X$ such that $X\cdot U\sim\N(0,1)$ which is constructed in Proposition~\ref{prop:decomp}. Using Theorem~\ref{thm:ghst}, we get
\begin{align*}
\widehat{\cZ}_{n,p}(\gb) = c_{p}\cdot\E_U \int_\dR \left(\E_{X} \exp\left(\sqrt{\frac{\gb}{n}}\cdot Xz U - \frac1{np}\abs{X z}^p\right)\right)^n \,dz.
\end{align*}
Using the fact that $X\sim\rho_p$, we simplify
\begin{align*}
&\widehat{\cZ}_{n,p}(\gb) \\
&= c_{p}\cdot\E_U \int_\dR(1+|z|^{p}/n)^{-n/p}\exp\left( n\psi_{p}\left(\sqrt{\gb/n}\cdot (1+|z|^{p}/n)^{-1/p}\cdot zU\right)\right) \,dz\\
&= c_{p}\cdot\E_U \int_\dR \exp\left(-\frac{n}{p}f(|z|^{p}/n)-\frac{|z|^p}{p}\cdot \frac{1}{1+|z|^p/2n}+ n\psi_{p}\left(\sqrt{\gb/n}\cdot\frac{zU}{ (1+|z|^{p}/n)^{1/p}}\right)\right) \,dz,
\end{align*}
where $$f(x)=\log(1+x)-2x/(2+x)$$ is a positive increasing function on $[0,\infty)$.

Note that, $$f'(x)=1/(1+x)-4/(2+x)^{2}=x^2/(1+x)(2+x)^{2}> 0\text{ for } x>0.$$
Heuristically, the term inside the exponential is approximately $-|z|^{p}/p +{\gb\nu_p^2}/{2}\cdot z^2U^2$. Thus, using the fact that $XU\sim\N(0,1)$, the upper bound is roughly 
\[
\E_{X,U}\exp({\gb\nu_p^2}/{2}\cdot X^2U^2)=(1-\gb\nu_{p}^{2})^{-1/2}<\infty \text{ when } \gb<\gb_{c}=\nu_{p}^{-2}.
\]
To rigorously upper bound the term, choose $\gd>0$ such that $\gb':=\gb(1+\gd)^{2/p}< \gb_c(p)=1/\nu_p^2$ and let $\widehat{\cZ}_{n,p}(\gb)=\sum_{i=0}^\infty J_i$ where
\begin{align*}
J_0&= c_{p}\cdot\E_U \int_{|z|^{p}\le 2n\gd}  K_{n,p,\gb}(z,U)\,dz,\\
J_i&= c_{p}\cdot\E_U \int_{2^i n\gd\le |z|^{p}\le 2^{i+1} n\gd} K_{n,p,\gb}(z,U) \,dz,\qquad i\ge 1,\\
K_{n,p,\gb}(z,U)&= \exp\left(-\frac{n}{p}f(|z|^{p}/n)- \frac{|z|^p}{p(1+{|z|^p}/{2n})}+ n\psi_{p}\left(\sqrt{\frac{\gb}{n}}\cdot\frac{zU}{ (1+|z|^{p}/n)^{1/p}}\right)\right).
\end{align*}
For $i=0$, we have
\begin{align*}
J_0\le  c_{p}\cdot\E_U \int_\dR \exp\left(-\frac{|z|^p}{p}\cdot \frac{1}{1+\gd}+n\psi_{p}\left(\sqrt{\gb/n}\cdot {zU}\right)\right) \,dz.
\end{align*}
Since $Z_p\cdot U=\eta\sim\N(0,1)$,
\begin{align*}
J_0\le (1+\gd)^{1/p}\E_{\eta}\exp\left( n\psi_{p}\left(\sqrt{\frac1n\gb(1+\gd)^{2/p}}\cdot {\eta}\right)\right)<\infty.
\end{align*}
In general, for $i\ge 1$ we get 
\begin{align*}
J_i 
&\le  c_{p}e^{-nf(2^{i}\gd)/p}\cdot\E_U\int_\dR \exp\left(-\frac{|z|^p}{p}\cdot \frac{1}{1+2^{i}\gd}+n\psi_{p}\left(\sqrt{\gb/n}\cdot\frac{zU}{ (1+2^{i}\gd)^{1/p}}\right)\right) \,dz\\
&\le 2^{1/p}e^{-(n-1)f(2^{i}\gd)/p}\E_{\eta} \exp\left( n\psi_{p}\left(\sqrt{\gb/n}\cdot\eta\right)\right).
\end{align*}
Thus, we get that $\widehat{\cZ}_{n,p}(\gb)  $ is uniformly bounded in $n$ when $\gb<\gb_{c}$.

To see the Gaussian fluctuation of the magnetization, let 
\begin{align*}
d\wh{\mu}_{n,\gb,p}
&:=\wh{\cZ}_{n,p}(\gb)^{-1} \bigg(\sum_{i=1}^n |x_i|^p/n\bigg)^{-1/p} \\
&\quad \exp\left(\frac{\gb}{2n}\cdot  \bigg(\sum_{i=1}^n x_i\bigg)^2\cdot \bigg(\sum_{i=1}^n |x_i|^p/n\bigg)^{-2/p}\right)\prod_{i=1}^n\rho_p(x_i)dx_i.
\end{align*}
For $u,v>0$, we consider
\begin{align*}
&\left\la\exp\left(-u \frac{S_n}{\sqrt{n}}-v\frac{T_n}{n}\right)\right\ra_{\wh{\mu}_{n,\gb,p}}\\
&\qquad=\wh{\cZ}_{n,p}(\gb)^{-1} \E\left[ \bigl(T_{n}/n\bigr)^{-1/p} \exp\left(\frac{\gb}{2n}\cdot  S_n^2\cdot \bigl(T_n/n\bigr)^{-2/p}-u \frac{S_n}{\sqrt{n}}-v\frac{T_n}{n}\right)\right].
\end{align*}
Using Theorem~\ref{thm:ghst}, we get
\begin{align*}
&\E\left[ \bigl(T_{n}/n\bigr)^{-1/p} \exp\left(\frac{\gb}{2n}\cdot  S_n^2\cdot \bigl(T_n/n\bigr)^{-2/p}-u \frac{S_n}{\sqrt{n}}-v\frac{T_n}{n}\right)\right]\\
&= c_{p}\cdot\E_U \int_\dR \left(\E_{X} \exp\left(\frac{X}{\sqrt{n}}\cdot (\sqrt{\gb}z U-u) - \frac{\abs{X}^p}{np}( \abs{z}^p+vp)\right)\right)^n \,dz\\
&= c_{p}\cdot\E_U \int_\dR\left(1+\frac{(|z|^{p}+vp)}{n}\right)^{-n/p}
\exp\left(n\psi_p \left(\frac{1}{\sqrt{n}}\cdot \frac{(\sqrt{\gb}z U-u)}{(1+(|z|^{p}+vp)/n)^{1/p}}\right)\right) \,dz\\
&= c_{p}\cdot\E_U \int_\dR \wt{K}_{n,p,\gb}(z,U)\,dz
\end{align*}
where $f(x)=\log(1+x)-2x/(2+x)$ and 
\begin{align*}
&\wt{K}_{n,p,\gb}(z,U)\\
&\quad =\exp\left(-\frac{n}{p}f\left(\frac{|z|^{p}+vp}{n}\right)- \frac{(|z|^p+vp)}{p\left(1+\frac{|z|^{p}+vp}{2n}\right)}
+n\psi_p\left( \frac{(\sqrt{\gb}z U-u)}{\sqrt{n}\left(1+\frac{|z|^{p}+vp}{n}\right)^{1/p}}\right)\right).
\end{align*}
Let $\gd>0$ be such that $\gb':=\gb(1+\gd)^{2/p}< \gb_c(p)$, then
\begin{align*}
&c_{p}\cdot\E_U \int_{\{|z|^{p}+vp\le 2n\gd\}}\wt{K}_{n,p,\gb}(z,U) \,dz\\
&\le  c_{p}\cdot\E_U \int_\dR \exp\left(-\frac{(|z|^p+vp)}{p}\cdot \frac{1}{1+\gd}+\frac{\nu_p^2}{2}\cdot (\sqrt{\gb}z U-u)^2\right) \,dz\\
&=c_p \cdot (1+\gd)^{1/p}\E_{U,X}\exp\left(-\frac{v}{1+\gd}+ \frac12\nu_p^2((1+\gd)^{1/p}\sqrt{\gb}X U-u)^2\right)<\infty.
\end{align*}
For each $i\ge 1$, let $a=2^{i}\gd$, then we get 
\begin{align*}
&c_{p}\cdot\E_U \int_\dR \ind_{\{na\le|z|^{p}+vp\le 2na\}}\wt{K}_{n,p,\gb}(z,U) \,dz\\
&\le  c_{p}e^{-nf(a)/p}\cdot\E_U\int_\dR \exp\left(-\frac{(|z|^p+vp)}{p}\cdot \frac{1}{1+a}+\frac{\nu_p^2}{2} (\sqrt{\gb}z U-u)^2 \right) \,dz\\
&\le 2^{1/p}e^{-(n-1)f(a)/p}\E_{U,X} \exp\left(-\frac{v}{1+a} + \frac12\nu_p^2(\sqrt{\gb}X U(1+a)^{-1/p}-u)^2\right).
\end{align*}
Since this is summable in $i$, we get that $\la\exp\left(-u S_n/\sqrt{n}-vT_n/n\right)\ra_{\wh{\mu}_{n,\gb,p}}$ is uniformly bounded in $n$ when $\gb<\gb_{c}$. 
Since $\lim_{t\to 0}\psi_p(t)/t^2 = \frac12 \nu_p^2 t^2$ by Lemma~\ref{lem:mgfbd}, we get
\begin{align*}
\lim_{n\to\infty}\left( 
\left(\frac{1}{\sqrt{n}}\cdot (1+(|z|^{p}+vp)/n)^{-1/p}\cdot (\sqrt{\gb}z U-u)\right)\right)^n 
=\exp\left( \frac{\nu_p^2}{2}(\sqrt{\gb}z U-u)^2 \right).
\end{align*}
Thus, it follows from the DCT that
\begin{align*}
&\lim_{n\to\infty}\E\left[ \bigl(T_{n}/n\bigr)^{-1/p} \exp\left(\frac{\gb}{2n}\cdot  S_n^2\cdot \bigl(T_n/n\bigr)^{-2/p}-\frac{u S_n}{\sqrt{n}}-\frac{vT_n}{n}\right)\right]\\
&= c_{p}\lim_{n\to\infty}\E_U \int_\dR\left(1+\frac{(|z|^{p}+vp)}{n}\right)^{-n/p}
\exp\left(n\psi_p \left(\frac{1}{\sqrt{n}}\cdot \frac{(\sqrt{\gb}z U-u)}{(1+(|z|^{p}+vp)/n)^{1/p}}\right)\right) \,dz\\ 
&= \E_{Z_2}\left[ \exp\left( -v+\frac{\gb\nu_p^2}{2}\left(Z_2 -\frac{u}{\sqrt{\gb}}\right)^2 \right)  \right]
= \frac{1}{\sqrt{1-\gb \nu_p^2}}\exp\left(-v+\frac{\nu_p^2}{2(1-\gb\nu_p^2)}u^2\right).
\end{align*}
Therefore, we conclude that 
\begin{align*}
\lim_{n\to\infty} \la\exp\left(-u S_n/\sqrt{n}-vT_n/n\right)\ra_{\wh{\mu}_{n,\gb,p}}
=\exp\left(-v+\frac{\nu_p^2}{2(1-\gb\nu_p^2)}u^2\right).
\end{align*}
The same argument works for pure imaginary exponents $u,v\in i \dR$. Thus, we see that $S_n/\sqrt{n}$ converges to $\N(0, \nu_p^2/(1-\gb\nu_p^2))$. Also $T_n/n$ converges to 1 under the measure $\wh{\mu}_{n,\gb,p}$, which yields the same convergence results under $\mu_{n,\gb,p}$.

\subsection{Proof of Theorem~\ref{thm:rhop} Part (b)}\label{sec:critical}

Consider the case where $\gb_c=\gb$ and $p\ge 2$. 
By Remark~\ref{rem:typical}, it suffices to prove the result for $\mvX=(X_1, X_2, \ldots,X_n)$ under the Gibbs measure~\eqref{eq:Gibbsmod}.
From Lemma~\ref{lem:condden} and CLT, we know that
\begin{align*}
\frac{\frac{1}{n}\sum_{i=1}^n Y_i-\E[Y_1]}{\sqrt{\var(Y_1)/n}}
&=
\frac{\frac{1}{n^{3/4}}\sum_{i=1}^n Y_i-n^{1/4}\E[Y_1]}{n^{-1/4}\sqrt{\var(Y_1)}}\\
&=
\frac{\frac{1}{n^{3/4}}\sum_{i=1}^n Y_i-\frac{1}{\sqrt{\gb_c}}\frac{Z^{(n)} U^{(n)}}{n^{1/4}}}{n^{-1/4}/\sqrt{\gb_c}} +o(1)
\implies \N(0,1).
\end{align*}
Thus, it suffices to consider the limit of $\frac{Z^{(n)} U^{(n)}}{n^{1/4}}$. Let $\phi:\dR\to\dR$ be a smooth, even function (that is, $\phi(t)=\phi(-t)$ for all $t\in\dR$) with compact support, then 
\begin{align*}
&\E[\phi(n^{-1/4}Z^{(n)} U^{(n)})]\\
&=\frac{2c_p}{\wh{Z}_{n,p}(\gb)}\int_0^\infty \int_0^\infty
\phi(n^{-1/4}zu)\\
&\qquad\qquad\exp\left( n\left( \psi_p\left(\sqrt{\frac{\gb}{n}}\frac{zu}{(1+\frac{|z|^p}{n})^{1/p}}\right)-\frac1p\log\left(1+\frac{|z|^p}{n}\right) \right) \right) \theta(u)\, dudz.
\end{align*}
By change of variable, $z\mapsto n^{\frac{1}{2p}}z$ and $u\mapsto n^{\frac14- \frac{1}{2p}}u$, the Taylor expansion of $\psi_p$, and Proposition~\ref{prop:decomp},
\begin{align*}
&\int_0^\infty \int_0^\infty
\phi(n^{-1/4}zu)
\exp\left( n\left( \psi_p\left(\sqrt{\frac{\gb}{n}}\frac{zu}{(1+\frac{|z|^p}{n})^{1/p}}\right)-\frac1p\log\left(1+\frac{|z|^p}{n}\right) \right) \right) \theta(u)\, dudz\\
&\qquad =\int_0^\infty \int_0^\infty
\phi(zu)
\exp\Bigg( 
\sqrt{n}\left( 
\frac12 z^2 u^2 - \frac1p z^p - \frac{p-2}{2p} u^{\frac{2p}{p-2}}
\right)\\
&\qquad\qquad-\left( 
\frac1p z^{p+2}u^2 + \frac{\kappa}{24}z^4 u^4 - \frac{1}{2p}z^{2p}
\right) + o(1)
\Bigg) 
g(n^{\frac{p-2}{4p}}u)n^{1/2} u^{\frac{p}{p-2}}\, dzdu
\end{align*}
where 
\[
\kappa_p=\gb_c^2 \cdot(3(\E Z_p^2)^2 - \E Z_p^4) = 3-\frac{\E Z_p^4}{(\E Z_p^2)^2} = 3-\frac{\Gamma(5/p)\gC(1/p)}{\Gamma(3/p)^2}>0.
\]
Here, we note that $-\gk_p$ is the excess kurtosis of the r.v.~$Z_p$. 
For fixed $u\ge 0$, let $z=u^{2/(p-2)}(1+n^{1/4} y)$, then
\begin{align*}
&\E[\phi(n^{-1/4}Z^{(n)} U^{(n)})]\\
& =\frac{2c_p}{\wh{Z}_{n,p}(\gb)}\int_0^\infty \int_{-n^{1/4}}^\infty
\phi(u^{p/(p-2)}(1+n^{1/4} y))\\
&\qquad \exp\Bigg( 
\sqrt{n}u^{2p/(p-2)}\left( 
\frac12 (1+n^{1/4} y)^2  - \frac1p (1+n^{1/4} y)^p - \frac{p-2}{2p} 
\right)\\
&\qquad-u^{4p/(p-2)}\left( 
\frac1p (1+n^{1/4} y)^{p+2} + \frac{\kappa}{24}(1+n^{1/4} y)^4 - \frac{1}{2p}(1+n^{1/4} y)^{2p}
\right) + o(1)
\Bigg) \\
&\qquad
g(n^{\frac{p-2}{4p}}u)n^{1/4} u^{\frac{p+2}{p-2}}\, dydu.
\end{align*}
Note that for fixed $u,y$, as $n\to\infty$, we get
\begin{align*}
\sqrt{n}u^{2p/(p-2)}\left( 
\frac12 (1+n^{1/4} y)^2  - \frac1p (1+n^{1/4} y)^p - \frac{p-2}{2p} 
\right)\to -\frac{p}{2}u^{2p/(p-2)}y^2.
\end{align*}
Let $\phi_0=1$, then it follows from Proposition~\ref{prop:decomp} and the DCT that
\begin{align*}
&\lim_{n\to \infty}\E[\phi(n^{-1/4}Z^{(n)} U^{(n)})]\\
&=\lim_{n\to \infty}\frac{\E[\phi(n^{-1/4}Z^{(n)} U^{(n)})]}{\E[\phi_0(n^{-1/4}Z^{(n)} U^{(n)})]}\\
&= \frac{1}{\cZ}\int_0^\infty \int_{-\infty}^\infty \phi(u^{p/(p-2)}) \exp\Bigg(  -\frac{p}{2}u^{2p/(p-2)}y^2 -u^{4p/(p-2)}\left(  \frac1p  + \frac{\kappa}{24} - \frac{1}{2p} \right) \Bigg)  u^{\frac{p+2}{p-2}}\, dydu\\
&= \frac{\sqrt{2\pi}}{\sqrt{p}\cZ}\int_0^\infty  \phi(u^{p/(p-2)}) \exp\Bigg( -u^{4p/(p-2)}\left(  \frac{1}{2p}  + \frac{\kappa}{24}  \right) \Bigg)  u^{\frac{2}{p-2}}\, du\\
&= \frac{\sqrt{2\pi}(p-2)}{p^{3/2}\cZ}\int_0^\infty  \phi(x) \exp\Bigg( -\left(  \frac1{2p}  + \frac{\kappa}{24} \right)x^4 \Bigg)  \, dx
\end{align*}
where 
\[
\cZ=\int_0^\infty \int_{-\infty}^\infty \exp\Bigg(  -\frac{p}{2}u^{2p/(p-2)}y^2 -u^{4p/(p-2)}\left(  \frac{1}{2p}  + \frac{\kappa}{24}  \right) \Bigg)  u^{\frac{p+2}{p-2}}\, dydu.
\]
For general smooth and compactly supported function $\phi$, we apply the same argument for $\wt{\phi}(t)=\phi(t)+\phi(-t)$. Therefore, we conclude that $\frac{1}{n^{3/4}}\sum_{i=1}^n Y_i$ converges in distribution to the distribution limit of  $\frac{1}{\sqrt{\gb_c}n^{1/4}}Z^{(n)} U^{(n)}$, whose density is given by
\[
C\exp\Bigg( -\gb_c^2\left(  \frac1{2p}  + \frac{\kappa}{24} \right)x^4 \Bigg)
\]
as desired.

\subsection{Proof of Theorem~\ref{thm:rhop} Part (c)}

By~\eqref{eq:ld-plimit} and Theorem~\ref{thm:genrho} (a), we see that the infimum of $I_F=I-F$ is strictly negative when $\gb>\gb_c$. Thus, $\lim_{n\to\infty}\frac1n \log\cZ_{n,p}(\gb)>0$.

We claim that 
\[
\psi_{p}\left(\sqrt{\gb}\cdot uv  \right)+\frac{1}{p}\log(1-u^{p})-\frac{p-2}{2p} v^{2p/(p-2)}
\]
has a unique maximizer $0<u_\ast<1,0< v_\ast<\infty$. The maximizers should satisfy
\begin{align*}
\psi_p'(\sqrt{\gb}uv)\sqrt{\gb}uv &= v^{2p/(p-2)},\text{ and }\\
\psi_p'(\sqrt{\gb}uv)\sqrt{\gb}uv &= \frac{u^p}{1-u^p}.
\end{align*}
From this, we obtain the equation
\begin{align}\label{eq:psimax}
\frac{(1+w^2)^{1/p}}{\sqrt{\gb} w\ }\cdot \psi_p'(\frac{\sqrt{\gb}w\ }{(1+w^2)^{1/p}})
=\frac{1}{\gb}(1+w^2)^{2/p},\qquad
w=v^{p/(p-2)}.
\end{align}
The left-hand side in~\eqref{eq:psimax} converges to $1/\gb_c$ as $w\to 0$ and the right-hand side equals $1/\gb$ when $w=0$. Since $\gb_c<\gb$, and $\frac{w}{(1+w^2)^{1/p}}, (1+w^2)^{2/p}$ are increasing in $w$ for $p\ge 2$, it follows from Lemma~\ref{lem:psidec} that there exists a unique nonzero solution $w_\ast$ to the equation~\eqref{eq:psimax}. Let $u_\ast=w_\ast^{2/p}(1+w_\ast^2)^{-1/p}$ and $v_\ast = w_\ast^{(p-2)/p}$, then 
\begin{align*}
\lim_{n\to\infty}\frac1n \log\cZ_{n,p}(\gb)
&=\sup_{0\le u< 1,w\ge 0 }\left\{ \psi_{p}\left(\sqrt{\gb}\cdot uv  \right)+\frac{1}{p}\log(1-u^{p})-\frac{p-2}{2p} v^{2p/(p-2)} \right\}\\
&= \psi_{p}\left(\sqrt{\gb}\cdot u_\ast v_\ast  \right)+\frac{1}{p}\log(1-u_\ast^{p})-\frac{p-2}{2p} v_\ast^{2p/(p-2)}.
\end{align*}

Let $(X_1, X_2 \ldots,X_n,Z^{(n)}, U^{(n)})$  be random variables with joint density
\[
f(x_1, x_2, \ldots,x_n,z,u)
=C \exp\left(\sqrt{\gb} zu \sum_{i=1}^n x_i - \frac{|z|^p}{p}\sum_{i=1}^n |x_i|^p +\log \theta(n^{\frac{p-2}{2p}}u)\right)\prod_{i=1}^n \rho_p(x_i).
\]
From the above computations, we see that $(Z^{(n)},U^{(n)})$ converges in distribution to a r.v.~supported on $(\pm w_\ast^{2/p}, w_\ast^{(p-2)/p})$ with probability $1/2$. Moreover, Lemma~\ref{lem:condden} tells us that $X_1, X_2, \ldots,X_n$ given $(Z^{(n)},U^{(n)})=(w_\ast^{2/p}, w_\ast^{(p-2)/p})$ are i.i.d.~with conditional density
\[
\exp\left(\sqrt{\gb} w_\ast \sum_{i=1}^n x_i - \frac{w_\ast^2}{p}\sum_{i=1}^n |x_i|^p-n\psi_p\left(\theta_\ast\right)\right) (1+w_\ast^2)^{n/p}\prod_{i=1}^n \rho_p(x_i).
\]
where $\theta_\ast = \sqrt{\gb}w_\ast(1+w_\ast^2)^{-1/p}$. By Lemmas~\ref{lem:psiibp} and~\ref{lem:condden}, we have
\begin{align*}
\E[X_1]
&=\int_{\dR} x\exp\left(\sqrt{\gb} w_\ast x - \frac{w_\ast^2}{p} |x|^p-\psi_p\left(\theta_\ast\right)\right) (1+w_\ast^2)^{1/p} \rho_p(x)\, dx\\
&=(1+w_\ast^2)^{-1/p} \int_{\dR} y\exp\left(\theta_\ast y - \frac{1}{p} |y|^p-\psi_p\left(\theta_\ast\right)\right) \, dy\\
&=(1+w_\ast^2)^{-1/p}\psi_p'\left(\theta_\ast\right)
= m_\ast,\\
\E[|X_1|^p] 
&=(1+w_\ast^2)^{-1} \int_{\dR} |y|^p\exp\left(\theta_\ast y - \frac{1}{p} |y|^p-\psi_p\left(\theta_\ast\right)\right) \, dy\\
&=(1+w_\ast^2)^{-1} (1+\theta_\ast\psi'_p(\theta_\ast))
=1.
\end{align*}
Here, we used the fact that $w_\ast$ is the unique positive solution to the equation~\eqref{eq:psimax}. From this and Remark~\ref{rem:typical}, we deduce that the magnetization is concentrated at $\E[X_1]= m_\ast$ when $Z^{(n)}\to w_\ast^{2/p}$. By conditioning on  $Z^{(n)}=-w_\ast^{2/p}, U^{(n)}=w_\ast^{(p-2)/p}$, we see that the magnetization is concentrated at $-m_\ast$ as desired.

\subsection{Proof of Theorem~\ref{thm:typical}}\label{ssec:typical}

Define 
\begin{align*}
(Y_1, Y_2, \ldots, Y_n) :\equald (X_1, X_2, \ldots, X_n) \mid (Z^{(n)}, U^{(n)})=(z,u)
\end{align*}
for $z\in \dR, u>0$ fixed. That is, the joint density of $Y_1, Y_2, \ldots, Y_n$ is given by
\begin{align*}
f_Y(y_1, y_2, \ldots,y_n)
&= \frac{f(y_1,\ldots,y_n, z,u)}{f_{Z^{(n)}, U^{(n)}}(z,u)}\\
&=\frac{\exp\left(\sqrt{\frac{\gb}{n}}\cdot  z u\sum_{i=1}^n y_i - \frac1{np}\abs{z}^p\sum_{i=1}^n |y_i|^p
\right)\, \prod_{i=1}^n \rho_p(y_i)}{\exp\left(n\left(
\psi_p\left(\sqrt{\frac{\gb}{n}}\frac{zu}{(1+\frac{|z|^p}{n})^{1/p}}\right)-\frac1p\log\left(1+\frac{|z|^p}{n}\right)
\right)\right)}.
\end{align*}
Since $Y_1, Y_2, \ldots,Y_n$ are i.i.d., the CLT tells us that
\begin{align*}
\frac{\frac{1}{\sqrt{n}}\sum_{i=1}^n Y_i-\sqrt{n}\E[Y_1]}{\sqrt{\var(Y_1)}}\implies \N(0,1).
\end{align*}
It is easy to see that $(Z_n,U_n)$ converges in distribution to a random vector $(Z,U)$ with joint density for $(Z,U)$ given by
\begin{align*}
f_{Z,U}(z,u) 
= C_1 \exp\left( \frac{\gb}{2\gb_c} z^2 u^2 -\frac{|z|^p}{p} \right) \theta(u),\qquad z\in\dR, u>0,
\end{align*}
where $C_1$ is the normalizing constant. Since  $\wh{\cZ}_{n,p}(\gb)$ is uniformly bounded in $n$ by the proof of Theorem~\ref{thm:rhop}, Part a, we can take a limit along a subsequence. From this and the fact that $X\cdot U\sim \N(0,1)$, one can see that the density of the product $Z\cdot U$ is
\begin{align*}
f_{Z\cdot U}(t)
=C_2\int_{0}^\infty \frac{1}{u}\exp\left(\frac{\gb}{2\gb_c}t^2\right)\exp\left(-\frac{|t|^p}{p u^p}\right)\theta(u)\, du = C_2 \exp\left(-\frac12 \left(1-\frac{\gb}{\gb_c}\right)t^2\right).
\end{align*}
Thus, $Z\cdot U \sim \N(0,\gb_c/(\gb_c-\gb))$.  By Lemma~\ref{lem:condden}, we have  $\lim_{n\to\infty}\sqrt{n}\E[Y_1]=\sqrt{\frac{\gb}{\gb_c^2}} \cdot ZU$ and $\lim_{n\to\infty}\var(Y_1) =\frac{1}{\gb_c}$. Thus, $\frac{1}{\sqrt{n}}(Y_1+Y_2+\cdots+Y_n)$ converges to a normal distribution with mean 0 and variance
\begin{align*}
\frac{1}{\gb_c}+\frac{\gb}{\gb_c^2}\frac{\gb_c}{\gb_c-\gb}
=\frac{1}{\gb_c-\gb}
\end{align*}
in distribution.

\subsection{Proof of Proposition~\ref{prop:selfnormal}}

Since
\[
\frac{n^{-1}\sum_{i=1}^{n}\gs_i^2}{\norms{\mvgs}_{n,p}^2}
\le n^{\tfrac{2}{p}-1},
\]
it suffices to consider the partition function
\begin{align*}
\widetilde{\cZ}_{n,p}(\gb) =
\E\left[\exp\left(\frac{\gb n}{2}\left(\frac{S_n}{T_n^{1/p}}\right)^2\right)\right]
\end{align*}
where $S_n = \frac{1}{n}\sum_{i=1}^n X_i$ and $T_n = \frac{1}{n}\sum_{i=1}^n |X_i|^p$. It was shown in~\cite[p.31]{PLS} that
\begin{align*}
\lim_{n\to\infty}\pr\left(\frac{S_n}{T_n^{1/p}}\ge x\right)^{1/n}
=\sup_{c\ge 0}\inf_{t\ge 0 }\E\left[\exp\left(t\left(cX - x \left(\frac{|X|^p}{p}+\frac{p-1}{p}c^{p/(p-1)}\right)\right)\right)\right]
\end{align*}
for all $x\ge 0$. Let $g(y)=\exp\left(\tfrac{\gb n}{2} y^2\right)$ and $Y_n=\frac{S_n}{T_n^{1/p}}$, then by symmetry we have
\[
\widetilde{\cZ}_{n,p}(\gb)
=\E[g(Y_n)]
=2\int_0^\infty g'(y)\pr(Y_n\ge y)\, dy.
\]
Thus, we get
\begin{align*}
&\lim_{n\to\infty}\frac1n \log\cZ_{n,p}(\gb) \\
&=\lim_{n\to\infty}\frac1n\log  \widetilde{\cZ}_{n,p}(\gb) \\
&=\sup_{y\ge 0 }\left\{
\frac{\gb y^2}{2}+
\sup_{c\ge 0}\inf_{t\ge 0 }\log\E\left[\exp\left(t\left(cX - y \left(\frac{|X|^p}{p}+\frac{p-1}{p}c^{p/(p-1)}\right)\right)\right)\right]
\right\}\\
&=\sup_{y,c\ge 0}\inf_{t\ge 0 }\left\{
\frac{\gb y^2}{2}+
\psi_{p}\left(\frac{ct}{(1+ty)^{1/p}}\right)-
\frac{1}{p}\log(1+ty)-\frac{p-1}{p}ty c^{p/(p-1)}
\right\}
\end{align*}
as desired.

\section{Proofs for the Generalized Hubbard--Stratonovich transform}\label{sec:GHSproof}

In this section, we prove the Generalized Hubbard--Stratonovich (GHS) transform (Theorem~\ref{thm:ghst}). The main ingredient of the GHS transform is the existence of a particular additive process, which will be proved in Theorem~\ref{thm:lgamma}. Based on the additive process, we construct a random variable $U_{p,q}$ such that $Z_p\cdot U_{p,q} = Z_q$ and provide a detailed analysis on the density of $U_{p,q}$ in Proposition~\ref{prop:decomp}.

Recall the collection of symmetric densities
\begin{align*}
\rho_{p}(x):=c_{p}\cdot e^{-\abs{x}^p/p},\quad x\in\dR
\end{align*} 
where $c_{p}:=\frac{p^{-1/p}}{2\gC(1+1/p)}$
for $p>0$. We use $Z_{p}$ to denote a typical random variable with density $\rho_{p}$.  Now, fix $p\ge 2$. Using Proposition~\ref{prop:decomp}, we can prove the GHS transform for $p>0$. 

\subsection{Proof of Theorem~\ref{thm:ghst}}

Recall that $Z_{2}\sim\N(0,1)$ satisfies $\E\exp(t Z_{2})=\exp(t^{2}/2)$ for all $t\in \dR$. By Proposition~\ref{prop:decomp}, there exist independent random variables $Z_{p}\sim \rho_{p}$ and $U\equald U_{p,2}$ such that $Z_{p}\cdot U\sim \N(0,1)$. Thus, we have
\begin{align*}
\exp\left(x^2y^{-2/p}\right)
&= \E\exp\left(\sqrt{2}xy^{-1/p}\cdot Z_2)\right)\\
&= \E_U \int_{\dR}c_{p}\exp\left(\sqrt{2}x\cdot y^{-1/p}z\cdot U -\abs{z}^p/p\right)\, dz\\
&= c_{p}y^{1/p}\E_U\int_{\dR}\exp\left(\sqrt{2}xzU-y\abs{z}^p/p\right)\, dz.
\end{align*}
Here we used a change of variable $z\to y^{1/p}z$ in the last equality.

\subsection{Proof of Theorem~\ref{thm:lgamma}}

Let $f:[0,\infty)\to[0,\infty)$ be a continuously differentiable, increasing function with $f(0)=0$. Let $(\cP_f, w_x)$ be a marked Poisson point process where $\cP_f$ is a Poisson point process with intensity $\mu_f(x) :=f'(x)/f(x)$, $x>0$, and the mark $w_x$ is an exponential random variable with mean $f(x)$. Define 
\begin{align*}
V_t^{(f)} 
= \sum_{x\in \cP_f\cap[0,t]}w_x
= \int_0^t w_x\, \cP_f(dx), \qquad t\ge 0,
\end{align*}
and $V_{t,s}^{(f)}:=V_s^{(f)} - V_t^{(f)}$ for $0\le t <s<\infty$. For an integer $k\ge 1$, let $f_k(x):=f(x/k)$ and $V^{(k)}_t$ be an independent copy of $V^{(f_k)}_t$. Note that
\begin{align*}
\E e^{i\xi V_t^{(f)}}
= \E\left(\prod_{x\in \cP_f\cap[0,t]}e^{i\xi w_x}\right)
&= \E\left( \exp\left( -\int_0^t \log(1-i\xi f(x))\,\cP_f(dx) \right) \right) \\
&= \exp\left( -\int_0^t \left(1-\frac{1}{1-i\xi f(x)}\right)\,\frac{f'(x)}{f(x)}\,dx \right)
= \frac{1}{1-i\xi f(t)}.
\end{align*} 
Define
\begin{align*}
U_t^{(m)}  := \sum_{k=1}^m (V_t^{(k)} - \E V_t^{(k)}).
\end{align*}
We assume that $\sum_{k=1}^\infty f(t/k)^2<\infty$ for each $t>0$. Then, $\E[U_t^{(m)}]=0$ and $\var(U_t^{(m)})<\infty$, which yields that $U_t^{(m)}\implies U_t$ as $m\to\infty$ in distribution, for all $t\ge 0$. Note that the characteristic function of $U_t$ can be computed as
\begin{align*}
\E e^{i\xi U_t}
=\prod_{k=1}^\infty\E e^{i\xi V_t^{(k)}}\cdot e^{-i\xi \E V_t^{(k)}} 
=\prod_{k=1}^\infty(1-i\xi f(t/k))^{-1}e^{-i\xi f(t/k)}.
\end{align*}
Let $f(x) := x/(1+x)$, then
\begin{align*}
\E e^{i\xi U_t}
=\prod_{k=1}^\infty\left(1-\frac{i\xi t}{k+t} \right)^{-1}e^{-i\xi t/(k+t)}.
\end{align*}
Note that 
\begin{align*}
e^{-\gc z}\prod_{k=1}^\infty (1+z/k)^{-1}e^{z/k}
=\frac{\gC'(z+1)}{\gC(z+1)}
= -\gc+\sum_{k=1}^\infty \frac{z}{k(k+z)},
\end{align*}
where $\gC$ is the Gamma function and $\gc$ is the Euler constant. It then follows that
\begin{align*}
\E e^{i\xi U_t}
&=\prod_{k=1}^\infty\left(1-\frac{i\xi t}{k+t} \right)^{-1}e^{-i\xi t/(k+t)}\\
&=\prod_{k=1}^\infty\left(1+(1-i\xi)t/k \right)^{-1}e^{(1-i\xi)t/k}\cdot \left(1+{t}/{k} \right)e^{-t/k} \cdot e^{i\xi t^2/(k(k+t))}\\
&=\frac{e^{\gc(1-i\xi)t}\gC((1-i\xi)t+1)}{e^{\gc t}\gC(t+1)} \prod_{k=1}^\infty e^{i\xi t^2/(k(k+t))}.
\end{align*}
Simplifying, we get 
\begin{align*}
\E e^{i\xi U_t}
&=\frac{\gC((1-i\xi)t+1)}{\gC(t+1)} \exp\left(i\xi t\left(-\gc+\sum_{k=1}^\infty \frac{t}{k(k+t)}\right)\right)\\
&=\frac{\gC((1-i\xi)t+1)}{\gC(t+1)} \exp\left(i\xi t\frac{\gC'(t+1)}{\gC(t+1)}\right)
=(1-i\xi)e^{i\xi}\cdot \frac{\gC((1-i\xi)t)}{\gC(t)} \exp\left(i\xi t\frac{\gC'(t)}{\gC(t)}\right).
\end{align*}
Let $X_t$ be the Gamma random variable with shape parameter $t$ and rate parameter $1$. Then
\begin{align*}
\E e^{-i\xi t \log X_t}
=\E X_t^{-i\xi t} = {\gC((1-i\xi)t)}/{\gC(t)},\quad
\E \log X_t = \frac{d}{du}\E e^{u\log X_t}|_{u=0} = {\gC'(t)}/{\gC(t)}.
\end{align*}
Let $Y$ be an exponential random variable with rate 1, independent of $(U_t)_t\ge 0$. Thus,
\begin{align*}
\E e^{i\xi (U_t+Y-1)}
&= e^{-i\xi}\cdot \E e^{i\xi Y}\cdot \E e^{i\xi U_t} \\
&= e^{-i\xi}\cdot (1-i\xi)^{-1}\cdot (1-i\xi)e^{i\xi}\cdot \frac{\gC((1-i\xi)t)}{\gC(t)}\cdot e^{i\xi t\gC'(t)/\gC(t)}\\
&= \frac{\gC((1-i\xi)t)}{\gC(t)}\cdot  e^{i\xi t\gC'(t)/\gC(t)}
=\E e^{-i\xi (t\log X_t-t\E \log X_t)}. 
\end{align*}
Thus, we have
\[
-U_t+1-Y \equald t\log X_t-t\E \log X_t, \text{ for all }t> 0.
\]
In particular, if we define the process
\begin{align}\label{eq:ytsum}
Y_t:= -U_t-Y + t\gC'(t+1)/\gC(t+1), \text{ for all } t\ge 0,
\end{align}
it has independent increments and $Y_t\equald t\log X_t$ for all $t> 0$.

From the proof, we can construct the marginals for the process $Y$ at any finitely many time points.  We get the existence of the process $Y$ over the positive real line by Theorem 9.7(ii) in~\cite[page 51]{satobook}, which states that if a system of probability measures $\{\mu_{s,t}:0\le s\le t<\infty\}$ satisfies
\begin{align*}
\mu_{s,t}\ast\mu_{t,u} &= \mu_{s,u}\text{ for all }0\le s\le t\le u<\infty,\\
\mu_{s,s}&=\delta_0 \text{ for } 0\le s<\infty,\\
\lim_{s\uparrow t}\mu_{s,t}&=\delta_0 \text{ for } 0< t<\infty,\\
\lim_{t\downarrow s}\mu_{s,t}&=\delta_0 \text{ for } 0\le s<\infty,
\end{align*}
then there exists an additive process $X_t$ such that the law of $X_t-X_s$ is $\mu_{s,t}$.

\subsection{Proof of Proposition~\ref{prop:decomp}}\label{sec:density}

According to the proof of Theorem~\ref{thm:lgamma}, for the existence of the density, it suffices to show that 
\begin{align*}
\prod_{k=1}^\infty\left(1-\frac{\i\xi t}{k+t} \right)^{-1}e^{-\i\xi t/(k+t)}
\end{align*}
converges absolutely. Indeed, let $a_k=\frac{\i\xi t}{k+t}$, then
\begin{align*}
\left|\frac{e^{-a_k}}{1-a_k}-1\right|
=\frac{|e^{-a_k}-1+a_k|}{1-|a_k|}
\le c\frac{|a_k|^2}{1-|a_k|}
\end{align*}
for large $k$ and it is summable. Thus, the infinite product converges absolutely. Since $Y_t$ is defined by the independent sum of nonnegative random variables (see~\eqref{eq:ytsum}) and one of them is an exponential random variable, it follows that the density is strictly positive. 

Let $p>q>0$ be fixed. Let $\theta:=\theta_{p,q}$ be the density of $U_{p,q}$. We can explicitly write down an expression for $\theta$ using M--Wright functions.  

Using Mellin transform and the facts that $\abs{Z_{p}}^p/p\sim \text{Gamma}(1/p,1)$, $\abs{Z_p}\cdot U_{p,q}\equald \abs{Z_q}$, one can check that the random variable $V:=p^{1/p}q^{-1/q}\cdot U_{p,q}$ has density given by
\[
f_{V}(x)=\frac{\gC(1/p)}{\gC(1/q)}\cdot \frac{1}{2\pi \i}\int_{1-\i\infty} ^{1+\i\infty}x^{-z}\cdot   \frac{\gC(z/q)}{\gC(z/p)}\,dz,\text{ for } x>0.
\]
Gamma function has simple poles at all non-positive integers $0,-1,-2,\ldots$. Thus, evaluating the contour integral as a sum over residues we get  for $x>0$,
\[
f_V(x) = \frac{\gC(1/p)}{\gC(1/q)}  \cdot \sum_{k=1}^{\infty} \frac{x^{qk}}{\gC(-\nu k)}\cdot \lim_{z\to -qk}(z+qk)\gC(z/q) = \frac{q\gC(1/p)}{\gC(1/q)}  \cdot \sum_{k=1}^{\infty} \frac{(-x^q)^k}{k!\cdot\gC(-\nu k)},
\]
where $\nu:=q/p\in(0,1)$. In particular, we have for $x>0$,
\[
\theta(x) = \frac{qp^{1/p}\gC(1/p)}{q^{1/q}\gC(1/q)}  \cdot \sum_{k=1}^{\infty} \frac{(-r/\nu)^k}{k!\cdot\gC(-\nu k)} = \frac{pc_q}{c_p}\cdot rM_{\nu}(r/\nu)
\]
where $r=x^{q}/p^{1-\nu}$ and $$M_{\nu}(x):=\sum_{n=0}^{\infty}\frac{(-z)^n}{n! \cdot \gC(1-\nu-\nu n)}$$ is the M--Wright function, a particular example of the Wright function of the second kind. Asymptotic representation of M--Wright functions for large arguments using a saddle-point analysis is well-known in the literature (see~\cite[eq.~4.3]{MMP10} and references within),
\begin{align}\label{eq:mwright}
M_\nu(r/\nu) \cdot r^{1-1/(2(1-\nu))}\cdot \exp((1-\nu)/\nu \cdot r^{1/(1-\nu)})  \to \frac{1}{\sqrt{2\pi(1-\nu)}}
\end{align}
as $r\to\infty$. Simplifying, we get the result.

\section{Proofs for the $\ell^p$ Self-Scaled Model }\label{sec:LargeDev}

In this section, we present the proofs of Theorems~\ref{thm:gen},~\ref{thm:genrho}, and~\ref{thm:genrho2}, where the limiting free energy of the $\ell^p$ self-scaled model is analyzed. Note that, Theorems~\ref{thm:rhop} and~\ref{thm:rhop12} follow from the proof of Theorem~\ref{thm:genrho}, and Theorem~\ref{thm:rhop2} follows from the proof of Theorem~\ref{thm:genrho2}.

\subsection{Proof of Theorem~\ref{thm:gen}}

Since $Q_n$ has the large deviation property (which follows from the moment assumption and~\cite[Theorem~II.4.1]{Ellis} and $\sup_{(x,y)\in\cX} F(x,y)<\infty$, it follows from~\cite[Theorem~II.7.1]{Ellis} that 
\[
\lim_{n\to\infty}\frac1n\log\cZ_{n,p}(\gb;\rho) = \lim_{n\to\infty}\frac1n\log\widetilde{\cZ}_{n,p}(\gb;\rho) = \sup_{0\le |x|\le y^{1/p}}\left(F(x,y)-I(x,y)\right),
\]
where $\widetilde{\cZ}_{n,p}(\gb;\rho)$ is defined as in~\eqref{eq:ztildegen}. Since $F$ is continuous on $\cX$, $\int_{\cX}e^{nF(x,y)}\, Q_n(dx,dy)$ is finite, and
\[
\lim_{L\to\infty}\limsup_{n\to\infty}\frac{1}{n}\log\int_{F\ge L}e^{nF(x,y)}\, Q_n(dx,dy)=-\infty,
\]
we apply~\cite[Theorem~II.7.2]{Ellis} to conclude the proof.

\subsection{Proof of Theorem~\ref{thm:genrho}}

\subsubsection{Part (a)}\label{sec:lowtemp}

Let $X,X_{1},X_{2},\ldots,X_{n}$ be i.i.d.~random variables from $\rho$. As in the proof of Theorem~\ref{thm:rhop} (a), we consider
\begin{align*}
\widehat{\cZ}_{n,p}(\gb;\rho)
= \E_\rho\left[\bigl(T_{n}/n\bigr)^{-1/p} \exp\left(\frac{\gb}{2n}\cdot  S_n^2\cdot \bigl(T_n/n\bigr)^{-2/p}\right)\right]
\end{align*}
where
\[
S_{n}:=\sum_{i=1}^{n}X_{i}\text{ and } T_{n}:=\sum_{i=1}^{n}\abs{X_{i}}^{p}.
\]
Then, it follows from Theorem~\ref{thm:ghst} that
\begin{align*}
\widehat{\cZ}_{n,p}(\gb;\rho)
&= c_p \iint \E_{U}\left[\exp\left(\sqrt{\gb n}S_n \cdot z\cdot U - T_n |z|^p/p\right)\right]\, dz\prod_{i=1}^n d\rho(x_i)\\
&= c_p \iint\left(\int \exp\left(\sqrt{\frac{\gb }{n}}x \cdot z\cdot u - \frac{1}{np}|x|^p |z|^p\right)d\rho(x)\right)^n \theta(u)\, du dz\\
&= c_p \iint \exp\left(n\psi_\rho(\sqrt{\gb}zu, -|z|^p/p)\right) \theta(n^{\frac12-\frac{1}{p}} u)n^{\frac12-\frac{2}{p}}\, du dz
\end{align*}
where $U\equald U_{p,2}$ is independent of $X_{i},i=1,2,\ldots,n$ such that $X\cdot U\sim\N(0,1)$. Since $\log\theta(n^{\frac12-\frac{1}{p}} u)=-\frac{p-2}{2p}n u^{2p/(p-2)}+o(1)$ by Proposition~\ref{prop:decomp}, we get
\begin{align*}
\lim_{n\to\infty}\frac{1}{n}\log \widehat{\cZ}_{n,p}(\gb;\rho)
=\sup_{z,u}\left\{\psi_{\rho}(\sqrt{\gb}zu,-|z|^p/p)-\frac{p-2}{2p}u^{\frac{2p}{p-2}}\right\}.
\end{align*}
Note that
\[
\pr_\rho(|X_1|\le x)\le \frac{\E_\rho|X_1|^{-\ga}}{x^{-\ga}}\le C x^{\ga}.
\]
Choose $k>\frac{pq}{\ga}$, then for all $n>k$, we have
\begin{align}\label{eq:Tnbdd}
\begin{split}
\E_{\rho}[|T_n|^{-q}]
&\le  \E_{\rho}[|\max_{i}|X_i|^{-pq}]\\
&= pq \int_0^{\infty} x^{-pq-1}\pr_\rho(|X_1|\le x)^n\, dx\\
&\le pq \int_0^{\infty} x^{-pq-1}\pr_\rho(|X_1|\le x)^k\, dx
\le C \int_0^{\infty} x^{-pq+\ga k-1}\, dx<\infty.
\end{split}
\end{align}
Let $\widehat{\gb}<\gb$ and $r=\gb/\widehat{\gb}$, then we have
\[
\log\widehat{\cZ}_{n,p}(\widehat{\gb};\rho)
=\frac{1}{r}\log \cZ_{n,p}(\gb;\rho) +\left(1-\frac{1}{r}\right)\log \E_{\rho}[|T_n|^{-q}]
\]
where $q=\frac{r}{p(r-1)}=\frac{\gb}{p(\gb-\widehat{\gb})}$. Letting $\widehat{\gb}\to \gb$ (that is, $r\to 1$),  it follows from the continuity of $\psi$ that 
\begin{align*}
\lim_{n\to\infty}\frac{1}{n}\log \cZ_{n,p}(\gb;\rho)
&=\lim_{\widehat{\gb}\to \gb}\lim_{n\to\infty}\frac{1}{n}\log\widehat{\cZ}_{n,p}(\widehat{\gb};\rho)\\
&=\lim_{\widehat{\gb}\to \gb}\sup_{z,u}\left\{\psi_{\rho}(\sqrt{\widehat{\gb}}zu,-|z|^p/p)-\frac{p-2}{2p}u^{\frac{2p}{p-2}}\right\}\\
&=\sup_{z,u}\left\{\psi_{\rho}(\sqrt{\gb}zu,-|z|^p/p)-\frac{p-2}{2p}u^{\frac{2p}{p-2}}\right\}.
\end{align*}
Define the function $f_\gb:[0,\infty)^2\to\dR$ given by
\[
f_\gb(z,w):= \psi_{\rho}(\sqrt{\gb}zw,-z^p/p) -\frac{p-2}{2p}\cdot w^{2p/(p-2)}.
\]
One can see that $\sup_{z,w\ge 0}f_\gb(z,w)> 0$ when $\gb> \gb_c(\rho)$. 
Indeed, fix $\gb>\gb_c(\rho)$, a constant $a>0$ and consider the function
\[
g(t)=p\cdot f_\gb(t^{2/p},at^{1-2/p}),\ t\ge 0.
\]
It is easy to check that $g(0)=g'(0)=0$ and 
\[
g''(0)=-2+p\gb \gs_X^2 \cdot a^2 - (p-2)\cdot a^{2p/(p-2)},
\]
where $\gs_X^2=\E[X^2]$. In particular, $g''(0)>0$ iff 
\[
\gb \gs_X^2 \cdot a^2 > 2/p + (1-2/p)\cdot a^{2p/(p-2)}.
\]
We can take $a_0=1$ so that 
\[
\frac{2/p + (1-2/p)\cdot a_0^{2p/(p-2)}}{a_0^2}=1<\gb \gs_X^2.
\]
Thus, there exists $t>0$ such that $f_\gb(t^{2/p},t^{1-2/p})>0$.

\subsubsection{Part (b)}

As in the proof of Theorem~\ref{thm:genrho}, it suffices to consider 
\begin{align*}
\widehat{\cZ}_{n,2}(\gb;\rho)
= \E_\rho\left[\bigl(T_{n}/n\bigr)^{-1/2} \exp\left(\frac{\gb}{2n}\cdot  S_n^2\cdot \bigl(T_n/n\bigr)^{-1}\right)\right].
\end{align*}
Using the GHS transform, we get
\begin{align*}
\widehat{\cZ}_{n,2}(\gb;\rho)
&= \E_\rho\left[\bigl(T_{n}/n\bigr)^{-1/2} \exp\left(\sqrt{\frac{\gb}{n}}\cdot  S_n\cdot \bigl(T_n/n\bigr)^{-1/2}\cdot \eta\right)\right]\\
&= c_2 \int \E_\rho\left[ \exp\left(\sqrt{\frac{\gb}{n}}\cdot  S_n\cdot y -\frac12 \frac{T_n}{n}y^2 \right)\right]\, dy\\
&= c_2 \sqrt{n} \int \left(\E_\rho\left[ \exp\left(\sqrt{\gb}\cdot  X\cdot y -\frac12 X^2\cdot y^2 \right)\right]\right)^n\, dy
\end{align*}
where $\eta\sim \N(0,1)$. Thus, we have
\begin{align*}
\lim_{n\to\infty}\frac{1}{n}\log \widehat{\cZ}_{n,2}(\gb;\rho)
=\sup_y \log \E_\rho\left[\exp\left(\sqrt{\gb} y\cdot X - \frac12 y^2 X^2\right)\right].
\end{align*}
Let $f(z)=\psi_{\rho}(\sqrt{\gb}z, -z^2/2)$. One can see that $\sup_{z\ge 0}f(z)>0$ for $\gb>\gb_c(\rho)=1$ because $f'(0)=0$ and $f''(0)=(\gb-1)\E|X|^2>0$. On the other hand, if $\gb<1$, then
\begin{align*}
\psi_{\rho}(\sqrt{\gb}z, -z^2/2) &=\log\E_\rho[\cosh(\sqrt{\gb}zX)\exp(-z^2 X^2/2)] \\
& \le\log\E_\rho[\exp((\gb-1)z^2 X^2/2)]\le 0
\end{align*}
for all $z\ge 0$. Thus, $\sup_{z\ge 0}f(z) = f(0)=0$.

\subsubsection{Part (c)}
By~\cite[p.31]{PLS} (as in the proof of Theorem~\ref{thm:rhop} (b)), we have
\begin{align}\label{eq:genrholim}
\lim_{n\to\infty}\frac1n \cZ_{n,p}(\gb;\rho) 
&=\lim_{n\to\infty}\frac1n \widetilde{\cZ}_{n,p}(\gb;\rho) \nonumber\\
&=\sup_{y\ge 0 }\left\{
\frac{\gb y^2}{2}+
\sup_{c\ge 0}\inf_{t\ge 0 }\log\E_\rho\left[\exp\left(t\left(cX - y \left(\frac{|X|^p}{p}+\frac{p-1}{p}c^{p/(p-1)}\right)\right)\right)\right]
\right\}\nonumber\\
&=\sup_{y\ge 0 }\left\{
\frac{\gb y^2}{2}+
\sup_{c\ge 0}\inf_{t\ge 0 }
\left\{
\log\E_\rho\left[\exp\left(ct X - ty \frac{|X|^p}{p}\right)\right]-\frac{p-1}{p}ty c^{p/(p-1)}
\right\}
\right\}\nonumber\\
&=\sup_{y,c\ge 0}\inf_{t\ge 0 }\left\{
\frac{\gb y^2}{2}+
\psi_{\rho}\left(ct, -{ty}/{p}\right)-\frac{p-1}{p}ty c^{p/(p-1)}
\right\}
\end{align}
where $\widetilde{\cZ}_{n,p}(\gb;\rho)$ is defined as in~\eqref{eq:ztildegen}.

Let $Z\sim \rho_p$, then by the GHS transform we have $Z \equald \eta \cdot U$ where $\eta\sim \N(0,1)$ and $\eta \perp U$. Define $\widetilde{U}=n^{\frac{p-2}{2p}}U\cdot \vone_{A_n}$ where $A_n=\{U\le k n^{\frac{2-p}{2p}}\}$ and $k>0$. Then,
\begin{align*}
\E[\widehat{\cZ}_{n,p}(\gb \widetilde{U}^2;\rho)]
&= \E\left[\bigl(T_{n}/n\bigr)^{-1/p} \exp\left(\frac{\gb}{2n}\cdot \widetilde{U}^2 \cdot  S_n^2\cdot \bigl(T_n/n\bigr)^{-2/p}\right)\right]\\
&= \E\left[\bigl(T_{n}/n\bigr)^{-1/p} \exp\left(\sqrt{\frac{\gb}{n}}\cdot \widetilde{U} \cdot  S_n\cdot \bigl(T_n/n\bigr)^{-1/p}\cdot \eta\right)\right]\\
&= \E\left[\bigl(T_{n}/n\bigr)^{-1/p} \exp\left(\sqrt{\gb}\cdot U \cdot  S_n\cdot \bigl(T_n\bigr)^{-1/p}\cdot \eta\right); A_n\right]\\
&\qquad+ \E\left[\bigl(T_{n}/n\bigr)^{-1/p} ; A_n^c\right]\\
&= \int \left(\E\left[ \exp\left(\sqrt{\gb} z X - \frac{1}{p}|z|^p |X|^p\right)\right]\right)^n\, dz
+ \E\left[\bigl(T_{n}/n\bigr)^{-1/p} ; A_n^c\right]\\
&\qquad+ \E\left[\bigl(T_{n}/n\bigr)^{-1/p} \exp\left(\sqrt{\gb}\cdot U\cdot \eta \cdot  S_n\cdot \bigl(T_n\bigr)^{-1/p} \right); A_n^c\right].
\end{align*}
By~\eqref{eq:Tnbdd}, we have
\begin{align*}
\E\left[\bigl(T_{n}/n\bigr)^{-1/p} ; A_n^c\right] \le  n^{\frac{1}{p}}\E\left[(T_{n})^{-1/p} \right]\le C n^{\frac{1}{p}}.
\end{align*}
Let $r>1$, then 
\begin{align*}
&\E\left[\bigl(T_{n}/n\bigr)^{-1/p} \exp\left(\sqrt{\gb}\cdot U\cdot \eta \cdot  S_n\cdot \bigl(T_n\bigr)^{-1/p} \right); A_n^c\right]\\
&\qquad \le n^{\frac{1}{p}}\E\left[(T_{n})^{-\frac{r}{p(r-1)}}\right]^{1-\frac{1}{r}} \E\left[\exp\left(r\sqrt{\gb}\cdot U\cdot \eta \cdot  S_n\cdot \bigl(T_n\bigr)^{-1/p} \right); A_n^c\right]^{1/r}.
\end{align*}
Since $S_n (T_n)^{-1/p}\le n^{1-1/p}$, we get 
\begin{align*}
\E\left[\exp\left(r\sqrt{\gb}\cdot U\cdot \eta \cdot  S_n\cdot \bigl(T_n\bigr)^{-1/p} \right); A_n^c\right]
\le \E_{\eta}\left[\int_{kn^{\frac{2-p}{2p}}}^\infty \exp\left(r\sqrt{\gb}n^{1-1/p} \eta x  +\log \theta(x)\right)\, dx\right].
\end{align*}
It follows from Proposition~\ref{prop:decomp} that 
\begin{align*}
r\sqrt{\gb}n^{1-1/p} \eta x  +\log \theta(x)
=r\sqrt{\gb}n^{1-1/p} \eta x -\frac{2-p}{2p}x^{\frac{2p}{2-p}}+ o(1),
\end{align*}
which implies that 
\begin{align*}
\E_{\eta}\left[\int_{kn^{\frac{2-p}{2p}}}^\infty \exp\left(r\sqrt{\gb}n^{1-1/p} \eta x  +\log \theta(x)\right)\, dx\right]<\infty.
\end{align*}
Thus,
\begin{align*}
\lim_{n\to\infty}\frac{1}{n}\log\E[\widehat{\cZ}_{n,p}(\gb \widetilde{U}^2;\rho)]
&=\lim_{n\to\infty}\frac{1}{n}\log\int \left(\E\left[ \exp\left(\sqrt{\gb} z X - \frac{1}{p}|z|^p |X|^p\right)\right]\right)^n\, dz\\
&=\lim_{n\to\infty}\frac{1}{n}\log\int\exp\left(n\psi_{\rho}(\sqrt{\gb}z,-|z|^p/p)\right)\, dz\\
&=\sup_{z}\psi_{\rho}(\sqrt{\gb}z,-|z|^p/p).
\end{align*}
Let $F_{n,p}(x;\rho):=\frac{1}{n}\log \widehat{\cZ}_{n,p}(x;\rho)$, then 
\begin{align}\label{eq:FnyU}
\sup_{z}\psi_{\rho}(\sqrt{y}z,-\frac{1}{p}|z|^p)
&=\lim_{n\to\infty}\frac{1}{n}\log \E\left[\exp\left(nF_{n,p}(y\widetilde{U}^2;\rho)\right)\right]\nonumber\\
&=\lim_{n\to\infty}\frac{1}{n}\log\int_{0}^{kn^{\frac{2-p}{2p}}}\exp\left(nF_{n,p}(y n^{\frac{2-p}{p}}z^2;\rho)+\log\theta(z)\right)\, dz\nonumber\\
&=\lim_{n\to\infty}\frac{1}{n}\log\int_{0}^{k}\exp\left(nF_{n,p}(y z^2;\rho)+\log\theta(n^{\frac{2-p}{2p}}z)\right)\, dz\nonumber\\
&=\lim_{n\to\infty}\frac{1}{n}\log\int_{0}^{k}\exp\left(n\left(F_{n,p}(y z^2;\rho)-\frac{2-p}{2p}z^{\frac{2p}{2-p}}\right)\right)\, dz.
\end{align}
By~\eqref{eq:genrholim}, we know that $F_{n,p}$ converges to a well-defined function $F_p$ as $n\to\infty$. Thus,  from equation~\eqref{eq:FnyU} we get that
\[
\sup_{z\ge 0 }\psi_{\rho}(\sqrt{y}z,-|z|^p/p)
=\sup_{x\ge 0}\left\{F_p(x^2 y ;\rho)-\frac{2-p}{2p}x^{\frac{2p}{2-p}}\right\}.
\]
Letting $y = \gb/x^2$, we get
\begin{align*}
F_p(\gb;\rho)\le \sup_{z\ge 0 }\psi_{\rho}\left({\sqrt{\gb}z}/{x},-|z|^p/p\right)+\frac{2-p}{2p}x^{\frac{2p}{2-p}}
\end{align*}
for all $x>0 $. Thus,
\begin{align*}
F_p(\gb;\rho)\le \inf_{x> 0 }\sup_{z\ge 0 }\left\{\psi_{\rho}\left({\sqrt{\gb}z}/{x},-|z|^p/p\right)+\frac{2-p}{2p}x^{\frac{2p}{2-p}}\right\}.
\end{align*}

\subsection{Proof of Theorem~\ref{thm:genrho2}}
Since
\[
\cZ_{n,p}(\gb;\rho) 
=\E\left[\frac{\gb}{n}\cdot \frac{\sum_{i<j}X_i X_j }{\norms{\mvX}^2_{n,p}}\right]
=\E\left[\frac{\gb}{n^{1-2/p}}\cdot \frac{\sum_{i<j}X_i X_j }{(\sum_i |X_i|^p)^{2/p}}\right]
\]
where $X_1, X_2, \ldots,X_n$ be an i.i.d.~sequence with common density $\rho$,
\[
\limsup_{n\to\infty}n^{1-2/p} \log\cZ_{n,p}(\gb;\rho)
\le \gb\sup\left\{\sum_{i<j}x_i x_j : \sum_{i=1}^\infty |x_i|^p=1\right\}.
\]
Let
\[
B(n,p):=\sup \left\{\sum_{1\le i<j\le n}x_i x_j : \sum_{i=1}^n |x_i|^p=1, x_i>0\right\}.
\]
Fix an integer $n\ge 2$. To find $B(n,p)$, we consider 
\[
F(x_1, x_2, \ldots,x_n,\lambda) = \sum_{1\le i<j\le n}x_i x_j - \lambda \left(\sum_{i=1}^n|x_i|^p - 1\right)
\]
and
\[
\frac{\partial}{\partial x_k}F(\mvx,\lambda)
=\sum_{i=1}^n x_i - (\lambda p x_k^{p-1}+x_k) = 0
\]
for all $k=1,2,\ldots,n$. If $\lambda =0$, then $\sum_i x_i = x_k$ for all $k$, which implies $x_k=0$ for all $k$. This is a contradiction, so $\lambda\neq 0$. In this case, for $c=\sum_{i}x_i$, $x_1, x_2, \ldots,x_n$ are solutions to the equation $\lambda p x^{p-1}+x-c=0$. Let $g(t)=\lambda p x^{p-1}+x-c$, then $g'(t) = \lambda p (p-1)x^{p-2}+1$. If $\lambda >0$, then there exists $t_\ast =(\lambda p (1-p))^{1/(2-p)}>0$ such that $g$ is decreasing on $(0,t_\ast)$ and increasing on $(t_\ast,\infty)$. Thus, there are at most two solutions to $g(t)=0$. If $\lambda<0$ then $f$ is increasing. Thus, there exists only one solution for $g(t)=0$. Therefore $B(n,p)$ attains its maximum when $\{x_i\}\subset\{a,b\}$ for some $a,b>0$.

Suppose that $B(n,p)$ attains its maximum when there are $k$ many $a$ and $l$ many $b$ among $x_1, x_2,\ldots, x_n$. That is, 
\begin{align*}
B(n,p)
=\max \left\{\sum_{1\le i<j\le n}x_i x_j : \sum_{i=1}^n |x_i|^p=1, x_i>0\right\}
=\binom{k}{2}a^2 + kl ab + \binom{l}{2}b^2
\end{align*}
given $ka^p+lb^p=1$. Without loss of generality, let $a>b$. Since
\begin{align*}
B(n,p)
=\frac{\binom{k}{2}a^2 + kl ab + \binom{l}{2}b^2}{(ka^p +l b^p)^{2/p}}
=\frac{\binom{k}{2} + kl (b/a) + \binom{l}{2}(b/a)^2}{(k +l (b/a)^p)^{2/p}}
\end{align*}
we consider
\begin{align*}
h_{k,l}(t) 
&= \log\left(\frac{k(k-1)+2kl t +l(l-1)t^2}{(k+lt^p)^{2/p}}\right)\\
&= \log(k(k-1)+2kl t +l(l-1)t^2) - \frac{2}{p}\log(k+lt^p)
\end{align*}
for $0<t<1$. Assume $k,l\ge 1$. To find the maximum of $h_{k,l}(t)$ in $t$, we consider
\begin{align*}
h_{k,l}'(t)
&= \frac{2kl+2l(l-1)t}{k(k-1)+2kl t + l(l-1)t^2} - \frac{2}{p}\frac{pl t^{p-1}}{k+lt^p}\\
&= \frac{2}{(k(k-1)+2kl t+ l(l-1)t^2)(k+lt^p)}((kl+l(l-1)t)(k+lt^p)-(k(k-1)\\
&\qquad +2kl t + l (l-1)t^2)lt^{p-1})\\
&= \frac{2kl}{(k(k-1)+2kl t +l(l-1)t^2)(k+lt^p)}(k+(l-1)t-(k-1)t^{p-1}-lt^p).
\end{align*}
Note that 
$$
k(k-1)+2kl t +l(l-1)t^2 = ((k-1)+(l-1)t)^2 + (k-1)+2(k+l-1)t+(l-1)t^2> 0.
$$
Let $u_{k,l}(t)=k+(l-1)t-(k-1)t^{p-1}-lt^p$, then 
\[
u_{k,l}''(t) = (1-p)t^{p-3}((p-2)(k-1)-plt)<0
\]
and so $u_{k,l}$ is concave. Since $u_{k,l}(1)=0$ and $\lim_{t\to 0}u_{k,l}(t)=-\infty$, there is at most one $t_0$ in $[0,1]$ such that  $u_{k,l}(t_0)=h_{k,l}'(t_0)=0$. Since $h_{k,l}(t)<0$ for $t\in[0,t_0)$ and $h_{k,l}(t)>0$ for $t\in(t_0,1]$ so $t_0$ is not a maximizer. Therefore, there is no maximizer $t$ in $[0,1]$. This implies that $k=0$ or $l=0$ and so $B(n,p) = \binom{n}{2}n^{-2/p}=\frac12 n^{1-2/p}(n-1)$.

Let $v_p(t)=(1-2/p)\log t + \log (t-1)$ for $t>1$. Then,
\begin{align*}
v_p'(t) 
=\left(1-\frac{2}{p}\right)\cdot \frac{1}{t}-\frac{1}{t-1}
=\frac{1}{pt(t-1)}((p-2)(t-1)+pt)
=\frac{1}{pt(t-1)}(2(p-1)t-(p-2)).
\end{align*}
Thus, $v_p(t)$ is increasing on $(1,(2-p)/(2(1-p)))$ and decreasing on $((2-p)/(2(1-p)),\infty)$. For each $k\ge 2$, we define $p_k\in(0,1)$ by $v_{p_k}(k)=v_{p_k}(k+1)$. Then, $p_k=2\log(1+1/k)/\log(1+2/(k-1))$ and
\[
\max\{v_p(n):n\ge 2 \text{ integer}\} = v_p(k)
\]
if $p\in[p_{k-1},p_k]$. 
Therefore, we conclude that 
\begin{align*}
\sup\left\{\sum_{i<j}x_i x_j : \sum_{i=1}^\infty |x_i|^p=1\right\}
&=\max\{B(n,p):n\ge 2 \text{ integer}\}= \frac{1}{2}k^{1-2/p}(k-1)
\end{align*}
if $p\in[p_{k-1},p_k]$. 

Now, it suffices to show that
\[
\liminf_{n\to\infty}n^{1-2/p} \log\cZ_{n,p}(\gb;\rho)
\ge \frac{\gb}{2}k^{1-2/p}(k-1).
\]
Fix $p\in[p_{k-1},p_k]$. Consider an i.i.d.~sequence of random variables, say $X_1,X_2,\ldots,X_n$, with common density $\rho$.\\

{\noindent\bf Under Assumption~\ref{assum1}}
Fix $C_1,C_2,C_3>0$ and define 
\begin{align*}
A :=\{ C_1 (n/k)^{1/p}\le X_{i}\le C_2 (n/k)^{1/p}\text{ for }1\le i\le k;  |X_{j}|<C_3 \text{ for }k<j\le n \}.
\end{align*}
On the event $A$,
\begin{align*}
\norms{\mvX}_{n,p}^2
=\left(\frac1n \sum_{i=1}^n |X_i|^p\right)^{2/p}
\le \left(C_2^p+C_3^p\right)^{2/p}
\end{align*}
and
\begin{align*}
\sum_{i<j}X_i X_j
&\ge \frac{k(k-1)}{2}C_1^2 (n/k)^{2/p}
- nk C_1 C_3 (n/k)^{1/p} -\frac12 C_3^2 n^2.
\end{align*}
So,
\begin{align*}
n^{1-2/p}\log\cZ_{n,p}(\gb;\rho)
&\ge  n^{1-2/p}\log\E\left[\exp\left(\frac{\gb}{n}\frac{\sum_{i<j}X_i X_j}{\norms{\mvX}_{n,p}^2}\right);A\right]\\
&\ge \frac{\gb}{(C_2^p+C_3^p)^{\frac{2}{p}}}\Bigg( \frac{k^{1-2/p}(k-1)}{2}C_1^2 - (nk)^{1-1/p}C_1 C_3 - \frac12 n^{2(1-1/p)}C_3^2 \Bigg)\\ 
&\qquad\qquad  +n^{1-2/p}\log\P(A).
\end{align*}
From Assumption~\ref{assum1}, choose $C_1=C_1(n),C_2=C_2(n)$ that satisfy $C_1,C_2\to\infty$, $C_1/C_2\to1$ as $n\to\infty$, and 
\[
\lim_{n\to\infty}n^{1-2/p}\left|\log\P\left(C_1 (n/k)^{1/p}\le X\le C_2 (n/k)^{1/p}\right)\right|=0.
\]
For fixed $C_3\in(0,\infty)$, we then have
\begin{align*}
\limsup_{n\to\infty} n^{1-2/p} |\log \P(A)|
&\le k\limsup_{n\to\infty} n^{1-2/p} \left|\log\P\left(C_1 (n/k)^{1/p}\le X\le C_2 (n/k)^{1/p}\right)\right|\\
&\qquad +\limsup_{n\to\infty} n^{2-2/p}
|\log \P(|X|\le C_3)|
= 0.
\end{align*}
Combining these, we get
\begin{align*}
\liminf_{n\to\infty} n^{1-2/p}\log\cZ_{n,p}(\gb;\rho) \ge \frac{C_1^2}{(C_2^p+C_3^p)^{2/p}}\cdot \frac{\gb}{2} k^{1-2/p}(k-1).
\end{align*}
Letting $n\to\infty$, we conclude that 
\[
\liminf_{n\to\infty}n^{1-2/p} \log\cZ_{n,p}(\gb;\rho)
\ge  \frac{\gb}{2}k^{1-2/p}(k-1).
\]
\smallskip

{\noindent\bf Under Assumption~\ref{assum2}}
For $C_1,C_2, C_3>0$, define 
\begin{align*}
B :=\{ C_1\le X_{i}\le C_2\text{ for }1\le i\le k;  |X_{j}|<C_3(k/n)^{1/p} \text{ for }k<j\le n \}.
\end{align*}
On the event $B$,
\begin{align*}
\norms{\mvX}_{n,p}^2
=\left(\frac1n \sum_{i=1}^n |X_i|^p\right)^{2/p}
\le \left(\frac{k}{n}(C_2^p +C_3^p)\right)^{2/p}
\end{align*}
and
\begin{align*}
\sum_{i<j}X_i X_j
&\ge \frac{k(k-1)}{2}C_1^2
- nkC_1 C_3 (k/n)^{1/p}-\frac12 n^2 C_3^2 (k/n)^{2/p}.
\end{align*}
So,
\begin{align*}
n^{1-2/p}\log\cZ_{n,p}(\gb;\rho)
&\ge  n^{1-2/p}\log\E\left[\exp\left(\frac{\gb}{n}\frac{\sum_{i<j}X_i X_j}{\norms{\mvX}_{n,p}^2}\right);B\right]\\
&\ge \frac{\gb}{(C_2^p+C_3^p)^{\frac{2}{p}}}\Bigg( \frac{k^{1-2/p}(k-1)}{2}C_1^2 - (nk)^{1-1/p}C_1 C_3 - \frac12 n^{2(1-1/p)}C_3^2 \Bigg)\\ 
&\qquad\qquad +n^{1-2/p}\log\P(B).
\end{align*}
First, there are $C_2(n)>C_1(n)>0$ such that  $C_2/C_1\to 1$ and
\[
\limsup_{n\to\infty}n^{1-2/p}|\log\P(C_1\le X\le C_2)|=0
\]
for any probability measure $\rho$. By Assumption~\ref{assum2}, we can choose $C_3$ such that $C_3\to 0$ and
\[
\limsup_{n\to\infty} n^{2-2/p} |\log \P(|X|\le C_3(k/n)^{1/p})|=0.
\]
Combining these, we get
\begin{align*}
\frac{C_1^2}{(C_2^p+C_3^p)^{2/p}}\to 1,\qquad
\limsup_{n\to\infty} n^{1-2/p} |\log \P(B)|\le 0.
\end{align*}
and so 
\begin{align*}
\liminf_{n\to\infty} n^{1-2/p}\log\cZ_{n,p}(\gb;\rho) \ge \frac{\gb}{2} k^{1-2/p}(k-1)
\end{align*}
as desired.

\section{Multi-species Curie--Weiss model under $\ell^p$ constraint}\label{sec:mscw}

One of the natural generalizations of the $\ell^p$ constraint Curie--Weiss model is a multi-group or multi-species version. A multi-group version of the Curie--Weiss model was introduced in~\cite{GC2008} and developed in~\cite{BRS2019, KLSS2020, KT2022, FCP2011}. In this model, there are $d$ many groups, and a spin interaction between groups $s$ and $t$ for $s,t\in[d]$ is given by $\gd_{s,t}$. One can study this model with $\ell^p$ ball spin configuration. 

To be precise, we consider the complete graph  with $N$ vertices. Suppose there are $d$ groups in the set of vertices, and each group has $N_1, N_2, \ldots, N_d$ many vertices. That is, $N=N_1+\cdots+N_d$. We assume that as $N\to\infty$, the ratio $\frac{N_s}{N}$ converges to $\ga_s$, with $\ga_1+\cdots+\ga_d=1$.  Let $I_s$ be the index set of $s$-th group (so $|I_s|=N_s$), then the Hamiltonian is defined by
\[
H_N(\mvgs)=\exp\left(\frac{1}{N}\sum_{s,t=1}^d \gd_{s,t}\sum_{i\in I_s}\sum_{j\in I_t}\gs_i\gs_j\right)
\]
for $\mvgs\in \gO_N=\{-1,1\}^N = \prod_{s=1}^d\{-1,1\}^{N_s}$ and a $d\times d$ inverse temperature matrix $\mvgd=(\gd_{s,t})_{s,t=1,\ldots,d}$. The phase transition of the magnetizations of each group $(S_1, \ldots, S_d)$, $S_t=\frac{1}{N_t}\sum_{i\in I_t}\gs_i$, has been studied in~\cite{FCP2011, KT2022}, which proved that the multi-group Curie--Weiss model has similar phase transition as in the classical model. The critical temperature is characterized by the matrix $J:=\mvgd^{-1}-\mvga$, where $\mvga$ is a diagonal matrix whose entries are $\ga_1,\ldots,\ga_d$. In the high-temperature regime where $J$ is positive definite, it was shown that the magnetizations $(S_1,\ldots, S_d)$ are concentrated on 0 with Gaussian fluctuation. 

One can consider the multi-group Curie--Weiss model under $\ell^p$ constraint by replacing the spin configuration space with $\Omega_{N,p}$  and the Gibbs measure with the uniform measure on $\Omega_{N,p}$. As above, one can rewrite the Hamiltonian and the Gibbs measure regarding the self-scaled vector with $\rho_p$ distribution. That is, the partition function has the following form
\[
\cZ= \int\exp\left(\frac{1}{N}\sum_{s,t=1}^d \gd_{s,t}\sum_{i\in I_s}\sum_{j\in I_t}\frac{x_i x_j}{\norms{\mvx}_{n,p}^2}\right)d\rho_p(\mvx).
\]
Applying the GHS transform to linearize the Hamiltonian as before, one can perform a similar analysis in the high-temperature regime and $p\ge 2$. Note that the Hubbard--Stratonovich transform method has been used in~\cite{KT2022} for multi-group Curie--Weiss model with $\pm 1$ spin case $\Omega_N$.

\section{Discussions and Further Questions}\label{sec:oq}

\begin{enumerate}[label={\arabic*}., leftmargin=*]
\item For $p\ge 2$ and $\gb<\gb_c(p)$ (the high-temperature regime), we have shown that the magnetization has a Gaussian fluctuation. The natural following question is to obtain a rate of convergence for CLT. This will be done in a forthcoming paper. 
\item Another generalization one can investigate is a multiple $\ell^p$ constraints version. One can consider the Curie--Weiss model, where the spin configuration is the intersection of several $\ell^p$ balls for different $p$'s. This has a natural connection to microcanonical ensemble models in~\cite{Nam20} and nonlinear Schr\"odinger equations.
\item One can generalize the $\ell^p$ constraint model to $p$-spin model, where the Hamiltonian is defined by the interaction between $k$ many spins for $k\ge 3$. 
\item It is interesting to study a geometric interpretation for the GHS transformation. To prove the GHS transform, we introduced a random variable $U_{p,q}$ independent of $Z_{p}$ in Proposition~\ref{prop:decomp} that satisfies $Z_{p}U_{p,q}\equald Z_{q}$. For the existence of such random variables, we constructed a stochastic process $U_{p,q}$  in $p$ for fixed $q$  in Theorem~\ref{thm:lgamma}, which is a stronger result. Since such stochastic processes can be considered a stochastic deformation of the uniform measures on $\ell^p$ balls in $p$, it is natural to ask if the process will provide any geometric interpretation of the GHS transform. 
\item 
Note the discontinuity of $\gb_c(p)$ at $p=\infty$. It is interesting to investigate the limiting behavior of the phase transition phenomena. In particular, it is expected that when $p=\ga\log n$, our model would be closely related to an extremal point process behavior as $n\to\infty$. One can ask if there is another phase transition with respect to $\ga$.
\end{enumerate}

\vspace*{1ex}
\noindent{\bf Acknowledgments.} 
We want to thank Grigory Terlov, Kesav Krishnan, and Qiang Wu for many enlightening discussions at the beginning stage of the project. We also thank Prof.~Renming Song and Prof.~Christian Houdr\'e for providing valuable references.

\bibliographystyle{chicago}
\bibliography{cwp.bib}

\end{document}